\documentclass[11pt,a4paper]{amsart}

\usepackage{soul}
\usepackage{graphicx}
\usepackage{subcaption}
\usepackage{enumerate}
\usepackage{amsmath}
\usepackage{amsfonts}
\usepackage{psfrag}
\usepackage{color}
\usepackage{pgf,tikz}
\usepackage{mathrsfs}
\usetikzlibrary{arrows}
\usepackage{amssymb}
\usepackage{mathtools} 
\usepackage{amsthm}
\usepackage{soul}
\usepackage{hyperref}

\newtheorem{theorem}{Theorem}

\newtheorem{definition}{Definition}
\newtheorem{remark}{Remark}

\newtheorem{lemma}{Lemma}
\newtheorem{proposition}{Proposition}
\numberwithin{equation}{section}
\allowdisplaybreaks[4]

\newcommand{\R}{\mathbb R}

\newcommand{\Tm}{{\mathbb T}_M}

\DeclareMathOperator \tv {TV}
\newcommand{\BV}{\hbox{\textrm{BV}}\,}
\newcommand{\N}{{\mathbb N}}

\newcommand{\Z}{\mathbb{Z}}

\renewcommand{\O}{\mathcal{O}}
\newcommand{\mm}{m}
\newcommand{\vv}{\mathtt v}
\DeclareMathOperator*{\essinf}{ess\,inf}

\newcommand{\DT}{{\Delta t}}
\newcommand{\be}{\begin{equation}}
\newcommand{\ee}{\end{equation}}
\def\RR{\mathsf{R}}
\def\MM{\mathsf{M}}

\newcommand{\modulo}[1]{{\left|#1\right|}}

\newcommand{\eps}{\varepsilon}

\definecolor{darkgreen}{rgb}{0.0, 0.5, 0.0}
\definecolor{ffqqqq}{rgb}{1.,0.,0.}
\definecolor{uuuuuu}{rgb}{0.26666666666666666,0.26666666666666666,0.26666666666666666}

\begin{document}

\title[]
{Euler-flocking system with nonlocal dissipation in 1D:\\ periodic entropy 
solutions}

\author[]
{D. Amadori, F. A. Chiarello}
\address{\newline 
Dipartimento di Ingegneria e Scienze dell'Informazione e Matematica (DISIM), University of L'Aquila -- L'Aquila, Italy}
\email{debora.amadori@univaq.it, felisiaangela.chiarello@univaq.it}

\author[]
{C. Christoforou}
\address{\newline Department of Mathematics and Statistics, University of Cyprus -- Nicosia, Cyprus}
\email{christoforou.cleopatra@ucy.ac.cy}

\date{\today}

\subjclass[2010]{Primary: 35L65; 35B40; Secondary: 35D30; 35Q70; 35L45.}

\vspace{-0.6cm}
\begin{abstract}
We consider a hydrodynamic model of flocking-type with all-to-all interaction kernel in a periodic domain in one-space dimension with linear pressure term. 
The main result is the global existence of periodic entropy weak solutions, for periodic initial data having finite total variation 
and initial density bounded away from zero. 
\end{abstract}
                  
\keywords{Euler-alignment system, periodic solutions}
                  
\maketitle    
                                
\vspace{0.2cm}

\setlength{\voffset}{-0in} \setlength{\textheight}{0.9\textheight}
\setcounter{page}{1} \setcounter{equation}{0}

\section{Introduction}
We are interested in weak solutions to the hydrodynamic model of flocking-type in one-space dimension that takes the form
\begin{equation}\label{eq:system}
\begin{cases}
    \partial_t\rho+\partial_x(\rho \vv )=0,\\
    \partial_t(\rho \vv)+\partial_x (\rho \vv^2 +p(\rho))=\int_{\mathbb{T}_L} \psi (x-y)\rho(x,t) \rho(y,t) (\vv (y,t)-\vv (x,t))dy
\end{cases}
 \end{equation}
where $(x,t)\in \mathbb{T}_\ell\times[0,+\infty)$, with $\mathbb{T}_\ell:=\R/{\ell \Z}$ and $\ell>0$. Here $\rho(x,t)$ represents the density, 
$\vv(x,t)$ the velocity, $\psi$ the alignment kernel and $p(\rho)$ the pressure term. We also assign the initial condition 
\begin{equation} \label{eq:initial_datum}
   (\rho(x,0),\vv(x,0))=(\rho_0(x),\vv_0(x)). 
\end{equation} 
and introduce $m=\rho \vv$ as the momentum with $m_0:= \rho_0 \vv_0$ its initial data. 

Self-organization is a interesting topic that has received a lot of attention in many respects including the mathematical modeling on this topic. 
Standard examples arise from phenomena in biology such as flock of birds, a swarm of bacteria or a school of fish, but one can also meet 
self-organized systems in other sciences as long as the emergence behavior is present in the system. 
Many mathematical models have been introduced to describe the emergence of flocking for such systems 
and most of them have arised from the pioneering work of Cucker and Smale~\cite{Cucker2007}. 
See~\cite{Carrillo2010, Shvydkoy2020} and the references therein. 
There is so far a broad range of studies on the subject from the discrete particle level up to the hydrodynamic formulation. 

In this paper, we study system~\eqref{eq:system} that has been derived in~\cite{Karper2014} as the hydrodynamic limit 
of the kinetic Cucker-Smale (CS) flocking model with the viscosity term and strong local alignment tending to zero. 
Actually the alignment term  was derived from the Motsch-Tadmor operator as a singular limit that improved CS model in small scales. 
In~\cite{Karper2014}, it is shown the convergence of weak solutions to the kinetic equation to strong solutions of the Euler system~\eqref{eq:system} 
with pressure 
\begin{equation*}
{\bf (P)}\qquad\qquad  p(\rho):=\alpha^2 \rho\,,\qquad    \alpha>0\;.
\end{equation*}
We remark that most studies on hydrodynamic models for flocking are on the pressureless Euler-alignment system.
We refer the reader to~\cite{Karper2013, Karper2014, Huang2006, Shvydkoy2021}  for results on the Euler-alignment system 
and point out that for the system with pressure, smooth and space-periodic solutions are obtained in~\cite{Choi2019}. 
In~\cite{AC2022, AC2024}, global entropy weak solutions to this pressured system have been established but in the setting of the Cauchy problem 
with initial data confined in a bounded domain and unconditional flocking has been captured in this weak framework.

The aim of this paper is to prove the existence of space-periodic entropy weak solutions to~\eqref{eq:system} defined for all times, 
for any initial data of bounded variation over the period and with initial density bounded away from zero.  We assume that
\begin{equation}\label{hyp-init_data}
\inf_{\mathbb{T}_\ell}\, \rho_0 >0\,,\qquad \tv\left\{\left(\rho_0, \vv_0\right); \mathbb{T}_\ell\right\}<+\infty,
\end{equation}
where $\tv$ denotes the total variation on the period.

Here we provide the definition of entropy weak solution to~\eqref{eq:system}-\eqref{eq:initial_datum}.
\begin{definition}~\label{entropy-sol}
Set $\Omega:= \mathbb{T}_\ell\times [0,+\infty)$ and assume $\left(\rho_0, \vv_0\right)\in L^\infty\left(\mathbb{T}_\ell\right)$ with $\rho_0(x)>0$ for a.e. $x$
and $\psi\in L^1_{loc}(\R)$.

Let  $(\rho, \vv):[0,+\infty)\times \mathbb{T}_\ell \to (0,\infty)\times \R$ and set $m=\rho \vv$.
We say that the $(\rho, \vv)$ is an entropy weak solution of the problem \eqref{eq:system}-\eqref{eq:initial_datum} 
if the following holds:
\begin{itemize}

\item the map $t\mapsto (\rho, \vv)(\cdot,t)\in L^1\left(\mathbb{T}_\ell\right)$ is continuous;

\item for every test function $\phi\in C^{1}_{\text{c}}(\Omega)$ 
it holds that
  \begin{align*}
      &\iint_{\Omega}\Big(\rho \,\phi_{t}(x,t)+ \rho \vv \, \phi_{x}(x,t)\Big) d x\, dt+\int_{\mathbb{T}_\ell} \rho_{0}(x)\phi(x,0) dx =0,\\
      &\iint_{\Omega}\Big(\rho\vv \,\phi_{t}(x,t)+ (\rho \vv^2+p(\rho))\, \phi_{x}(x,t)\Big) d x\, dt\\
      &\qquad -\iint_{\Omega}\Big(\int_{\mathbb{T}_\ell} \psi (x-y)\rho(x,t) \rho(y,t) (\vv (y,t)-\vv (x,t))dy\Big) \phi(x,t) d x\, dt\\
      &\qquad +\int_{\mathbb{T}_\ell} m_{0}(x)\phi(x,0) dx =0;
  \end{align*}
\item for every pair of smooth entropy--entropy flux $(\eta, q)$,  
that is,
\begin{equation*}
    \left(-\vv^2 + p'\right)\eta_{m} = q_\rho\,,\qquad 
    \eta_\rho + 2\vv \eta_m = q_{m}\,,
\end{equation*}
with $\eta$ convex, it holds
\begin{equation}\label{eq:entropy}
\partial_t \eta(\rho,m) + \partial_x q(\rho,m)\le \eta_{m} \int_{\mathbb{T}_\ell} \psi(x-y)\left(\rho(x,t)m(y,t)  - \rho(y,t)m(x,t) \right) dy
\end{equation}
in the sense of distributions on $\Omega$.
\end{itemize}
\end{definition}
From Definition~\ref{entropy-sol} one can prove that an entropy weak solution conserves the total mass 
and the total momentum 
over the period ${\mathbb{T}_\ell}$, i.e. defining
\begin{equation*}
    \mathcal{M}(t):=\int_{\mathbb{T}_\ell} \rho(x,t)\,dx,\quad \mathcal{M}_1(t):=\int_{\mathbb{T}_\ell} m(x,t)\,dx, \qquad 
    t\ge 0, 
\end{equation*}
we get 
\begin{equation*}
    \mathcal{M}(t)=\mathcal{M}(0)=\int_{\mathbb{T}_\ell}\rho_0(x)\,dx,\qquad \mathcal{M}_1(t)=\mathcal{M}_1(0)=\int_{\mathbb{T}_\ell}m_0(x)\,dx, \qquad \forall\, t\ge 0,
\end{equation*}
where we used that Definition \ref{entropy-sol} implies $\rho(\cdot,0)=\rho_0,\,m(\cdot, 0)=m_0.$
Let $M$ and $M_1$ denote the constant values of $\mathcal{M}(t)$ and $\mathcal{M}_1(t)$ respectively.

Throughout this paper the 
all-to-all interaction is assumed, i.e. the kernel $\psi$ is constant. Without loss of generality we assume:
\begin{equation*}
{\bf (\Psi)}\qquad\qquad  \psi(x)\equiv 1\,.
\end{equation*}

The main result of this paper, which is the global existence of entropy weak solutions to~\eqref{eq:system}-\eqref{eq:initial_datum}, is stated here.
\begin{theorem}\label{Th-1} 
Assume {\bf (P)}, $\bf{(\Psi)}$ and that the initial data satisfy \eqref{hyp-init_data}.
Then the Cauchy problem \eqref{eq:system}-\eqref{eq:initial_datum} admits an entropy weak solution $(\rho,\vv)$ in the sense of Definition~\ref{entropy-sol}, with $(\rho,\vv)(\cdot, t)\in \BV(\mathbb{T}_\ell)$ $\forall\, t,$ and it holds  
\begin{equation*}
	\tv\left\{\left(\rho,\vv\right)(t); \mathbb{T}_\ell\right\}\le C,\qquad \forall \,\,t>0,
\end{equation*}
\par\noindent
where  $C=C(q)$ and 
$q\doteq\frac{1}{2}\tv{\{\ln(\rho_0)\}}+\frac{1}{2\alpha}\tv{\{\vv_0\}}\,.$ 
\end{theorem}

The analysis follows the strategy developed in~\cite{AC2022} and adapted to the periodic case. 
We transform the problem in Lagrangian coordinates in the spirit of Wagner~\cite{WAGNER1987118} 
and construct a front-tracking approximate solution similarly as in~\cite{AC2022}. 
The presence of vacuum in~\cite{AC2022} raised additional obstacles that are not here, such as loss of strict hyperbolicity 
and concentration terms of Dirac deltas on the extended momentum in the notion of weak solution. 
In this paper, the periodicity of the problem allows the wave fronts when touching the boundaries in $\mathbb{T}_\ell$ to enter the domain from the other side, 
while in~\cite{AC2022}, the boundaries in Lagrangian coordinates act as non-reflecting, and, therefore, the fronts there are absorbed from the boundaries. 
These are some of the differences between the construction here and in~\cite{AC2022}. In addition, the proof of the $L^\infty$ bounds on the approximate solutions require a different approach, compared to \cite{AC2022}.

It is worth spending few words on some papers in literature, such as \cite{Choi2019}, in which the author studies the global existence and uniqueness of strong solutions for the compressible isothermal Euler system with  nonlocal dissipation under suitable smallness and regularity assumptions on the initial data, and its large-time behavior showing the fluid density and the velocity converge to its averages exponentially fast as time goes to infinity. 
In \cite{Ewelina24} the authors study an Euler compressible system characterized by the presence on the right-hand side of a damping term, a nonlocal alignment term and a nonlocal repulsion/attraction forces term. The goal of that paper is to establish the global-in-time existence of finite-energy entropy solutions for this system without restriction on the size of initial data, using the vanishing viscosity limit for the strong solutions of the compressible Navier-Stokes equations. 
Moreover, in a very recent work \cite{Choi2024} the authors rigorously derive the isentropic Euler-alignment system with singular communication weights, considering a kinetic BGK-alignment model consisting of a kinetic BGK-type equation with a singular Cucker-Smale alignment force.
{In the present paper we focus on establishing the global existence of entropy weak solutions that are periodic with initial data away from the vacuum 
and allowing any size in $\BV$ norm. In \cite{CalvoColomboFrid2008}, the authors analyse the well-posedness of spatially periodic solutions of specific 
relativistic isentropic gas dynamics equations. Other insights about periodic solutions can be found in \cite{Dafermos2016}.}

The structure of the paper is as follows. First, in Section~\ref{S2}, we transform the problem into Lagrangian coordinates 
in the periodic domain $\mathbb{T}_M$  with $M$ being the total mass on $\mathbb{T}_\ell$ and state the global existence theorem 
for the new variables $(u,v)$ (Theorem~\ref{newthm}). 
We achieve the existence result in Section~\ref{Sect:2} in several steps. First, we construct an approximate sequence of solutions in~\ref{S2.3}, 
then, we introduce the Glimm-type linear functional in Subsection~\ref{S2.4} and show that this is non-increasing in time. 
Next, we establish uniform $L^\infty$ bounds on the approximate solution. In Subsection~\ref{subsect:Lxi}, we introduce the weighted functional $L_\xi$  
that helps us prove the estimates on the vertical traces in Subsection~\ref{subsec:vert-traces}. The proof of Theorem~\ref{newthm} is then in Subsection~\ref{S2.8}. 
Last, Section~\ref{Sect:3} is devoted to reformulate the previous results in terms of the Eulerian variables and to prove Theorem~\ref{Th-1}. 

\section{Set up in Lagrangian coordinates}\label{S2}
First, due to the conservation of mass and momentum, and hypothesis~${\bf (\Psi)}$, system  \eqref{eq:system} becomes
\begin{equation*}
\begin{cases}
\partial_t\rho+\partial_x(\rho \vv )=0,\\
\partial_t(\rho \vv)+\partial_x (\rho \vv^2 +p(\rho))= \rho M_1- M \rho \vv,
\end{cases}
\end{equation*}
Without loss of generality, we can reduce to the case $M_1=0$ thanks to the following change of the independent variables $(x,t)$ and of $\vv$:
	$x \mapsto x-\bar \vv t,\quad \vv\mapsto \vv-\bar \vv$
where $\bar \vv=M_1/M.$ This allows to reduce to the case $M_1=0,$ because in these new variables, that we call again $(x,t),$ the average of the momentum is zero, i.e.
	\begin{equation}
	\int_{\mathbb{T}_\ell} m(x,t) dx=0,\quad \forall t\geq 0.
	\end{equation}
Hence we get
\begin{equation}\label{eq:system2}
\begin{cases}
\partial_t\rho+\partial_x(\rho \vv )=0,\\
\partial_t(\rho \vv)+\partial_x (\rho \vv^2 +p(\rho))=-M \rho \vv\,,
\end{cases}
\end{equation}
with initial condition
\begin{equation} \label{eq:initial_datu-bar-v}
   (\rho(x,0),\vv(x,0))=(\rho_0(x),\vv_0(x) - \bar \vv)\,. 
\end{equation} 
In these new variables, the entropy inequality \eqref{eq:entropy} in Definition~\ref{entropy-sol} rewrites as 
\begin{equation}\label{entropy_eulerian}
\partial_t \eta(\rho,m)+\partial_x q(\rho,m)\leq -M \eta_{m} m . 
\end{equation}
Next, we reformulate the system  from Eulerian to Lagrangian coordinates. More precisely, we use the mapping
\begin{align*}
(x,t) &\mapsto (y=\chi(x,t),\,\tau=t),\\
\chi(x,t) &=\int_0^x \rho(x',t)\,dx'\in [0,M]\,,\qquad { x\in [0,\ell]},
\end{align*}
and set the variables
\begin{equation}\label{def-rho-v}
\displaystyle   u(y,t)=\dfrac{1}{\rho(x,t)},\qquad\qquad v(y,t)=\vv(x,t).
\end{equation}
Then the derivatives satisfy 
\begin{equation*}
    \partial_\tau=\partial_t+\vv\partial_x,\qquad\qquad \partial_y=\tfrac{1}{\rho}\partial_x=u\partial_x,\,
    \end{equation*}
and system~\eqref{eq:system2} reduces to 
\begin{equation}\label{eq:system_Lagrangian_M} 
\begin{cases}
\partial_\tau u-\partial_y v=0,\\
\partial_\tau v + \partial_y (\alpha^2/u)=-Mv,
\end{cases}
\end{equation}
with $(y,t)\in \Tm\times [0,\infty).$  By periodicity, we have $(u,v)(0+,t)=(u,v)(M-,t).$  
From \eqref{eq:initial_datu-bar-v}, the initial condition is given by
\begin{equation}\label{eq:init-data-uv}
(u,v)(\cdot,0) = (u_0,v_0) := (\rho_0^{-1},\vv_0 - \bar \vv)\,.
\end{equation}

From now on,  we will consider the interval $[0,M)$ representing the torus $\mathbb{T}_M$.

Recalling \eqref{hyp-init_data}, the initial data $\rho_0$ is bounded from above and away from 0; 
this ensures that the change of variable $$x\mapsto \chi(x,0)= \int_{0}^x \rho_0(x')\,dx'$$ 
is bi-Lipschitz continuous from $[0,\ell)$ to $[0,M)$\,.

Therefore, the initial condition \eqref{eq:init-data-uv}, together with \eqref{hyp-init_data},
lead to:
\begin{equation}\label{eq:init-data-lagr}
(u_0,v_0)\in \BV(\mathbb{T}_M)\,, 
\qquad \essinf_{\mathbb{T}_M} u_0 >0\,, \qquad  \int_{\mathbb{T}_M} v_0(y)\, dy=0\,, \qquad
{\int_{\mathbb{T}_M} u_0(y)\, dy=\ell}\,.
\end{equation}
We have thus reformulated our problem in Lagrangian variables and the aim is to construct an entropy weak solution  $(u(y,t),v(y,t))$ to the  problem~\eqref{eq:system_Lagrangian_M}-\eqref{eq:init-data-uv} that conserves mass and momentum. 
By entropy weak solution to \eqref{eq:system_Lagrangian_M}-\eqref{eq:init-data-uv} we mean 

\begin{definition}~\label{entropy-sol_lagr}
 Assume $\left(u_0, v_0\right)\in L^\infty\left(\mathbb{T}_M\right)$ with $u_0(y)>0$ for a.e. $y$.
	Assume that $\psi\in L^1_{loc}(\R)$.
		Let  $(u, v):\mathbb{T}_M \times [0,+\infty) \to (0,\infty)\times \R.$
	We say that $(u, v)$ 
	is an entropy weak solution of the problem \eqref{eq:system_Lagrangian_M}-\eqref{eq:init-data-uv}
	if the following holds:
	\begin{itemize}
			\item the map $t\mapsto (u, v)(\cdot,t)\in L^1\left(\mathbb{T}_M\right)$ is continuous;
			\item $(u,v)(\cdot,0)=\displaystyle{\lim_{t\to 0^+}}(u,v)(\cdot,t)=(u_0, v_0);$
			\item $(u,v)$ is a distributional solution to \eqref{eq:system_Lagrangian_M}-\eqref{eq:init-data-uv};
		\item for every pair of smooth entropy--entropy flux $(\tilde \eta, \tilde q)$,  
		with $\tilde \eta$ convex, it holds
		\begin{equation}\label{eq:entropy_lagr}
		\partial_t \tilde\eta(u,v) + \partial_x \tilde q(u,v)\le -M\tilde \eta_{v} v,
		\end{equation}
		in the sense of distributions on $\mathbb{T}_M\times (0,+\infty)$.
	\end{itemize}
\end{definition}

Thus the proof of the following existence result for the problem~\eqref{eq:system_Lagrangian_M}-\eqref{eq:init-data-uv} is the theme of this section.

\begin{theorem}\label{newthm}
Consider the problem~\eqref{eq:system_Lagrangian_M}-\eqref{eq:init-data-uv} and assume \eqref{eq:init-data-lagr}. 
Then, there exists an entropy weak solution $(u,v)$ in the sense of Definition \ref{entropy-sol_lagr}.
Moreover, there exist positive constants $u_{inf}$, $u_{sup}$, $\widetilde C_0$ such that for all $t>0$
\begin{equation}\label{eq:Linfty-bound-uv-new}
0<u_{inf}\le u(y,t) \le u_{sup}\,,\quad 
    |v(y,t)|\le \widetilde C_0
\qquad \mbox{for a.e.}\ y \,{ \in\mathbb{T}_M}\;.
\end{equation}   
\end{theorem}

The following analysis prepares the ground for the proof of this theorem that is given in Subsection~\ref{S2.8}.

\subsection{Characteristic curves}\label{subsec:2.1}
The characteristic speeds of system~\eqref{eq:system_Lagrangian_M} are given by
\begin{equation}\label{eq:charact_speed}
\lambda_{1,2}=\pm \frac{\alpha}{u}, \qquad u>0,
\end{equation}	 
with the indices corresponding to each family $j=1,2$. The rarefaction-shock curves, issued at a point $(u_l,v_l)$ are:
\begin{itemize}
	\item curve of the first family
	\begin{equation}\label{sys:first_family}
	\begin{cases}
	u<u_l, \quad v=v_l- \alpha \left( \sqrt{\frac{u_l}{u}-\frac{u}{u_l}} \right),\\
    u>u_l, \quad v=v_l+\alpha \ln \left(\frac{u}{u_l}\right);
	\end{cases}
	\end{equation}
	\item curve of the second family
	\begin{equation} \label{sys:second_family}
	\begin{cases}
	u<u_l, \quad v=v_l- \alpha \ln \left(\frac{u}{u_l}\right),\\
	u>u_l, \quad v=v_l- \alpha \left( \sqrt{\frac{u}{u_l}-\frac{u_l}{u}} \right).
	\end{cases}
	\end{equation}
\end{itemize}
The curves \eqref{sys:first_family}, \eqref{sys:second_family} can be parametrized as follows:
\begin{equation*}
u \to v(u;u_l,v_l)=v_l+2\alpha \, h(\varepsilon_j),\quad u>0,\quad j=1,2
\end{equation*}
with 
\begin{equation*}
\varepsilon_1=\frac{1}{2}\ln\left(\frac{u}{u_l}\right),\qquad \varepsilon_2=\frac{1}{2}\ln\left(\frac{u_l}{u}\right)\,,\qquad 
h(\varepsilon)=\begin{cases}
\varepsilon &\varepsilon\geq 0,\\
\sinh \varepsilon &\varepsilon <0.
\end{cases}
\end{equation*}
We refer to $\varepsilon_j,\,j=1,2$ as the strength of a wave of the $j-$characteristic family. 

For any pair of left and right states  $(u_l,v_l)$, $(u_r,v_r)$, with $u_{l,r}>0$, the Riemann problem to the homogeneous system 
of~\eqref{eq:system_Lagrangian_M} has a unique solution that consists of simple Lax waves and the 
corresponding strengths  $\eps_j$ be the strength of the $j=1$, $2$ wave satisfy
\begin{align}\nonumber 
\eps_2 - \eps_1 &= \frac 12 \ln \left(\frac{u_l}{u_r} \right) =  \frac 12 \ln \left(\frac{p_r} {p_l}\right)\\
\nonumber
h(\eps_1) + h(\eps_2) &= \frac{v_r - v_l}{2\alpha}
\end{align} 
where $p_{r,l} = p(u_{r,l}) = \alpha^2 / u_{r,l}$ and 
\begin{equation}\label{eq:bound-on-size-of-Rp}
|\eps_1| + |\eps_2|\le \max\left\{ \frac 12 \left| \ln \left(\frac{u_r} {u_l}\right)\right|, \frac{|v_r - v_l|}{2\alpha} \right\}\,.
\end{equation}
See for instance ~\cite[Proposition 2.1]{AC2022}.

\section{Global existence of weak solutions}\label{Sect:2}
In the following subsections, we lay the groundwork for the proof of Theorem~\ref{newthm}. 
More precisely, we construct approximate solutions to system~\eqref{eq:system_Lagrangian_M} using the front-tracking algorithm 
in conjunction with the operator splitting, introduce Lyapunov functionals that allow us to obtain uniform bounds on the total variation 
and prove uniform $L^\infty$ bounds on $(u,v)$ as well. The compactness of the sequence yields the convergence to the limit 
which is the solution to~\eqref{eq:system_Lagrangian_M}.

\subsection{Front-tracking approximate solutions}\label{S2.3} 
In this subsection we construct the approximate solutions $(u^\nu, v^\nu)$, for $\nu\in\mathbb{N}$, 
to the initial value problem~\eqref{eq:system_Lagrangian_M} with $M$-periodicity, following the approach adopted in~\cite{AC2022} for a free-boundary problem.
The construction combines the wave-front tracking algorithm with the operator splitting scheme to treat the source term. 

Below we describe the construction in three steps: 

\medskip
\paragraph{{\bf Step 1.}}\quad 
Fix $\nu\in\N$, and choose approximate periodic initial data $(u^\nu_0, v^\nu_0)$, which are piecewise constant and  satisfy
\begin{equation*}
\tv \{\left(u^\nu_0, v^\nu_0\right); \mathbb{T}_M\} \le \tv \{\left(u_0, v_0\right); \mathbb{T}_M\}\,,\qquad \| \left(u^\nu_0, v^\nu_0\right) -  \left(u_0, v_0\right)\|_\infty \le \frac 1 \nu\,,
\end{equation*}
and assign these as the initial data:
$$u^\nu(y,0) := u_0^\nu(y), \quad v^\nu(y,0) := v_0^\nu(y).$$

As a consequence of \eqref{eq:init-data-lagr}, 
one has that
\begin{equation}\label{int-v0-to-0}
\left|\int_{\mathbb{T}_M} v^\nu_0(y)\,dy\right| \le \frac M\nu \to 0 \qquad \nu\to\infty\,,
\end{equation}
and
\begin{equation}\label{int-u0-nu}
{ \essinf_{(0,M)} u^\nu_0 \ge \essinf_{(0,M)} u_0 - \frac 1\nu >0\,,\qquad \left|\int_{\mathbb{T}_M} u^\nu_0(y)\,dy - \ell\right| \le \frac M\nu \to 0 \qquad \nu\to\infty\,.}
\end{equation}

\medskip
\paragraph{{\bf Step 2.}}\quad Fix two parameters: the time step $\Delta t=\Delta t_\nu>0$ and the threshold $\eta=\eta_\nu>0$.
Having those, we set the times $t^n=n\Delta t$, $n=0,1,2,\dots$ 
and we use the parameter $\eta$ to control the size of rarefaction fronts and the errors in the speeds of shocks and rarefaction fronts. This would be made clear in the next step. Allow now both parameters $\Delta t_\nu$, $\eta_\nu$ to converge to 0 as $\nu\to\infty$.

\medskip
\paragraph{{\bf Step 3.}}\quad The approximate solution $(u^\nu, v^\nu)$ of~\eqref{eq:system_Lagrangian_M} is obtained as follows:

\begin{enumerate}
    \item[(i)] On each time interval $[t^{n-1},t^n)$, the approximate solution $(u^\nu, v^\nu)$ is the $\eta$-front tracking approximate solution to the homogeneous system of~\eqref{eq:system_Lagrangian_M} with initial data $$(u^\nu(y,t^{n-1}+), v^\nu(y,t^{n-1}+)).$$
    We recall that $\eta$-front tracking approximate solutions (see Def. 7.1 in \cite{Bressan_Book}) are piecewise constant functions with discontinuities along finitely many lines in the  $(y,t)$ half-plane and interactions between two incoming fronts. Indeed, by modifying the wave speeds slightly, by a quantity less than $\eta$, we can assume that there are only interactions of two incoming fronts. 
    To give more details of this construction, solve the Riemann problems at $t=t^{n-1}$ around each discontinuity point of the initial data $(u^\nu(y,t^{n-1}+), v^\nu(y,t^{n-1}+))$. 
    Then, retain shocks as obtained in the Riemann solution, but approximate rarefactions by fans of fronts each one of them consisting of  two constant states connected by a single discontinuity of strength less than $\eta$. More precisely, each rarefaction of strength $\varepsilon$ is approximated by $N$ rarefaction fronts, with $N=[\varepsilon/\eta]+1$, each one of strength $\varepsilon/N<\eta$ and with speed to be equal to the characteristic speed of the right state.  
   Then prolong the approximate solution within $[t^{n-1}, t^n)$ until some wave fronts interact at a point, or some wave reaches one of the boundaries of $\mathbb{T}_M,\,y=0$ and $y=M$. 

  \smallskip
 - At the interaction, there are exactly two incoming fronts as mentioned and the approximate solution is prolonged by solving the Riemann problem and approximating rarefactions, if they arise in the solution, as described  above. 
   
     \smallskip
  - At times $t=\tau$ that a front meets the boundaries, i.e. $y=0$ or $y=M$, then that wave front enters back into the domain from the other boundary $y=M$ or $y=0$, respectively,  for $t=\tau+$.
  
  \smallskip
    \item[(ii)]  At the time $t=t^n$, apply the operator splitting technique to system~\eqref{eq:system_Lagrangian_M} and define
\begin{equation}\label{eq:u-v_fractional-step}
u^\nu(y,t^n+) = u^\nu(y,t^n-)\,,\qquad  v^\nu(y,t^n+) = v^\nu(y,t^n-) \left(1-M\Delta t\right)
\end{equation}
taking into account the source term. As in (i), by possibly changing  slightly the speed of some fronts, 
there is no interaction of fronts at the times steps $t^n$. 
  \smallskip
\item[(iii)] Proceed by iteration of (i) and (ii), with $n\in\mathbb{N}$.
 \end{enumerate}

The approximate solution is defined as long as the iteration at Step 3 can be carried out and in particular, as long as the interactions do not accumulate. 
As in \cite{AC2022}, this version of front tracking algorithm for $2\times2$ systems does not require the presence of non-physical fronts:
as long as the solution exists, waves fronts do not accumulate. 

\subsection{The linear functional}\label{S2.4} We define the Lyapunov functional associated with the approximate solution $(u^\nu,v^\nu)$ to system~\eqref{eq:system_Lagrangian_M}. For convenience, we drop the index $\nu$ from here and on.

Let $t$ be a time different from any $t^n$ and at which no wave interaction occurs, and let $\{y_j\}_1^N$ be the discontinuity points of $(u,v)(\cdot,t)$ in $[0,M)$, with
\begin{equation}\label{eq:disc-points}
  0\le y_1< y_2<\ldots< y_{N(t)}<M\,,
  \end{equation}
for some integer $N=N(t)$. Let $\varepsilon_j$ be the corresponding strength of the front located at $y_j$.  Then we define the \emph{linear functional}
 \begin{equation}\label{def:L}
 L(t) = \sum_{j=1}^{N(t)} |\varepsilon_j|\,.
 \end{equation}

Similarly to~\cite[Lemma 3.1]{AC2022}, the following relation between $L(t)$ and the total variation in  $y$ of the approximate solution holds.

 \begin{lemma}\cite[Lemma 3.1]{AC2022}\label{prop:equivalence-Lin}
 The following identities hold,
 \begin{equation}\label{eq:identity-tvlnu}
 \frac 12 \tv \{\ln(u)(\cdot,t); [0, M)\} =  \frac 12 \tv \{\ln(\sigma(u(\cdot,t));[0, M)\} = L(t)\,,
 \end{equation}
 where $\sigma$ is defined $\sigma(u)=-p\left(\frac{1}{u}\right)=-\frac{\alpha^2}{u}$, and
 \begin{equation}\label{eq:tv-v}
 \tv \{v(\cdot,t);[0, M)\}\le 2\alpha \cosh(c)\, L(t),
 \end{equation}
 where $c=\max{|\varepsilon|}$ is the maximum of the sizes of the waves on $[0,M)$ at time $t$.
 \end{lemma}

Now, we study the variation of $L(t)$ and show bounds on the approximate solution valid for all times and independent of $\nu$. 
For this purpose, we adopt the following notation: the variation $\Delta G$  of a function of time $t\mapsto G(t)$
is the change of the function $G$ around $t$, i.e. $\Delta G(t)=G(t+)-G(t-)$. Similarly, for the variation in space $y$.

Here we recall some estimates on the variation of the wave strenghts at interactions and at time steps, see \cite[Lemma 5.4]{ABCD_JEE_2015} and {\cite[Proposition 3.1-3.2]{AC2022}}.

\begin{proposition}\label{S2:prop:estimate-time-step}

(a) Consider the interaction of two waves of the same family of sizes $\alpha$ and $\beta$,
resulting into a wave of the same family with size $\eps$ and a reflected wave of size $\eps_{refl}$. Then 
\begin{equation}\label{eq:Delta-L}
|\eps| + |\eps_{refl}| \le |\alpha|+ |\beta|\,.
\end{equation}

(b) Let $\varepsilon^-$ and $\varepsilon^+$ be the sizes of a wave before and after a time step $t^n$, respectively, 
and let $\varepsilon_{refl}$ be the size of the reflected wave. Then 
\begin{equation}\label{Delta-L-time-step}
|\eps^+| + |\eps_{refl}| = |\eps^-|
\end{equation}

If $q>0$ satisfies $|\varepsilon^-|\leq q$, then
\begin{equation}\label{time-step-reflected-wave}
c_1(q) {{M}\Delta t|\varepsilon^-| } \,\le\, |\varepsilon_{refl}| \,\le\, C_1^\pm(q) {{M}\Delta t|\varepsilon^-| }
\end{equation}
where
\begin{equation*}
c_1(q)= \frac1{1+\cosh(q)}\,,\qquad C_1^\pm(q) = 
\begin{cases} 
\frac{1} {2}  & \mbox{ if }\eps^->0\\
\frac{\cosh(q)} {2} & \mbox{ if }\eps^-<0\,.
\end{cases}
\end{equation*}
\end{proposition}

\smallskip
Now, we can establish the uniform bound on the linear functional $L(t)$, as long as the approximate solution is defined.

\begin{lemma}\label{lem:bounds-on-bv}
Suppose that the approximate solution $(u^\nu, v^\nu)=(u, v)$ is defined for $t\in[0,T]$. 
Then $L(t)$ is non-increasing in time and
\begin{equation}\label{LLin-decreases}
L(t)\le L(0+) \le  q,\qquad
\forall\, t\in[0,T]\,,
\end{equation}
where
\begin{equation}\label{def:q}
q\doteq\frac{1}{2}\tv{\{\ln(\rho_0)\}}+\frac{1}{2\alpha}\tv{\{\vv_0\}}\,.
\end{equation}
Moreover, there exists a constant $C_1>0$ independent of $\nu$ and $t$ such that
\begin{equation}\label{eq:unif-bound-tv-v}
\tv \{v(\cdot,t); \mathbb{T}_M\} \le C_1 \,, \qquad t\in[0,T]\,.
\end{equation}
\end{lemma}
\begin{proof}
Similarly to~\cite{AC2022}, the linear functional is non-increasing:
\begin{equation}\label{L-decreases}
\Delta L(t)\le 0 
\end{equation}
at all times the approximate solution is defined, i.e. in between the time steps and at the time steps $t=t^n=n\Delta t$.
Indeed, the possible cases are:

$\bullet$ 
when a 1-front reaches $y=0$ at time $t$, then, by construction, it enters $[0,M)$ from the point $y=M$ at the time $t+$ retaining the same strength, and therefore $\Delta L(t)=0$; this is a consequence of periodicity. A similar behavior holds at $y=M$;

$\bullet$ at an interaction point, that is, a point $(\bar y,t)$ with $\bar y\in [0,M)$ where two waves interact, the following holds: 
if the two waves belong to different families, then they cross each other without changing their strengths and in such cases $\Delta L(t)=0$; 
while if they are of the same family, 
then by \eqref{eq:Delta-L} we conclude that $\Delta L(t)\le 0$. Notice that this may happen also at $\bar y=0$;

$\bullet$ last, at the time steps $t^n$, the identity \eqref{Delta-L-time-step} implies $\Delta L(t^n)=0$.

\smallskip
Combining the above, we get immediately that $L(t)$ is non-increasing. 
At time $t=0+,$ using \eqref{eq:bound-on-size-of-Rp}  
we have
\begin{equation*}
L(0+)\le  \frac 12 \tv \{\ln(u_0)\} +  \frac 1{2\alpha} \tv \{v_0\} =q
\end{equation*}
that depends on $\nu$ through $v=v^\nu,\,u=u^\nu,$ 
and the proof of \eqref{LLin-decreases} follows. In view of \eqref{LLin-decreases} and taking $c=q$ in~\eqref{eq:tv-v}, we obtain
\begin{equation}\label{eq:tv-v_m}
\tv \{v(\cdot,t)\} \le 2\alpha \cosh(q)\, L(t)
\end{equation}
Thus, inequality \eqref{eq:unif-bound-tv-v} holds with $C_1 = 2\alpha \cosh(q)$. The proof is now complete.
 \end{proof}

\subsection{Uniform bounds and global existence} \label{S3.4}

In the next lemma, we show that the supremum and infimum of $u$ are comparable with a factor depending on $q$.

\begin{lemma} \label{lem:sup-inf-u}
For every $t\geq 0,$ one has 
\begin{equation}\label{sup-inf-u}
     \sup  \{u(\cdot,t) \}\leq \inf \{u(\cdot,t) \}\cdot e^{2 q
     }\,.   
\end{equation}
\end{lemma}
\begin{proof}
By \eqref{eq:identity-tvlnu} and \eqref{LLin-decreases} we deduce that
\begin{align*}
    \sup \{\ln u(\cdot,t) \}- \inf \{\ln u(\cdot,t)\} \le \tv \{\ln u(\cdot,t) \}\le 2 L(t) \le 2 L(0) \le 2q
\end{align*}
which yields \eqref{sup-inf-u}.
\end{proof}
From Lemma~\ref{lem:sup-inf-u}, we deduce that 
\begin{equation*}
    \lim_{t\to+\infty}\inf \{u (t,\cdot) \}=0 \implies  \lim_{t\to+\infty}\sup \{u (t,\cdot)\} =0
\end{equation*}
and 
\begin{equation*}
    \lim_{t\to+\infty}\sup \{u (t,\cdot)\} =+\infty \implies  \lim_{t\to+\infty}\inf \{u (t,\cdot)\} =+\infty\,,
\end{equation*}
 i.e. the sup and inf of $u$ are comparable.

Now, we turn our attention to the integral of $(u^\nu,v^\nu)$ over $[0,M)$. Let us define
 the quantities 
\begin{equation}\label{cal V and U}
{\mathcal V}^\nu(t) = \int_0^M v^\nu(y,t)\,dy\,,\qquad 
{\mathcal U}^\nu(t) = \int_0^M u^\nu(y,t)\,dy \,
\end{equation}
and derive bounds for these in the next lemma.

\begin{lemma}\label{L4}
For every $t\geq0,$ one has 
\begin{align}\label{estimate:mathcalV}
    \modulo{\mathcal{V}^\nu(t)}&\leq e^{-M t} e^{M \Delta t} \cdot \frac{M}{\nu}+ \frac{C}{M} \eta +  C \eta \Delta t\\
    \label{estimate:mathcalU}
    {\modulo{\mathcal{U}^\nu(t) -\ell}}&\leq   \frac{M}{\nu}+  \tilde C  \, \eta_\nu\int_0^t\inf \{u^\nu(\tau)\}\, d\tau
\end{align}
for suitable constants $ C,\,\tilde C >0,$ which  are independent of $t$ and $\nu$.
\end{lemma}

\begin{proof} The proof of \eqref{estimate:mathcalV} is analogous to the one in \cite[Lemma 4.4]{AC2022}. 
We observe that $\mathcal{V}^\nu (t)$ is piecewise linear on each $(t^n,t^{n+1})$ and it is possibly discontinuous at 
at each $t^n$; indeed, \eqref{eq:u-v_fractional-step} yields
\begin{equation}\label{Vnu-tn}
\mathcal{V}^\nu(t^n+)=(1-M \Delta t) \mathcal{V}^\nu(t^n-)\,.
\end{equation}
Differentiating~\eqref{cal V and U}$_1$, one obtains
\begin{align*}
\frac{d}{dt} {\mathcal V}^\nu(t) &=  -\sum_{j=0}^{N^\nu(t)+1} \Delta v^\nu (y^\nu_j) (y^\nu_j)' = (I) + (II)
\end{align*}
where 
$$
(I) =- \sum_{j=0}^{N^\nu(t)+1} \Delta v^\nu (y^\nu_j) \mu^\nu_j \;,
\quad 
(II) = -\sum_{j=0}^{N^\nu(t)+1} \Delta v^\nu (y^\nu_j) \left[ (y^\nu_j)' - \mu^\nu_j \right],
$$
with $\mu^\nu_j$ given by the exact  Rankine-Hugoniot speed.
We observe that $(I)=0$ because it is a  telescopic sum and thanks to the periodicity. 
Indeed  the Rankine-Hugoniot condition (RH)  for \eqref{eq:system_Lagrangian_M} gives us
\begin{align*}
    \Delta v^\nu (y^\nu_j) \mu^\nu_j =\Delta (\alpha^2/u^\nu)(y^\nu_j), \quad j=1,...,N^\nu(t).
\end{align*}
By means of \eqref{eq:tv-v_m}, we obtain
\begin{align*}
|(II)| &\le \eta_\nu \tv \{v^\nu, \mathbb{T}_M\}\\
& \le \eta_\nu 2\alpha \cosh(q)\, L(0+)\,.    
\end{align*}
Hence for $t'<t''$, with $t',\,t''\in (t^{n-1},t^n),$ we deduce
\begin{equation*}
    \modulo{\mathcal{V}^\nu(t'')-\mathcal{V}^\nu (t')} \leq  C  \eta_\nu (t''-t')
\end{equation*}
with $C=2\alpha q\cosh(q).$ On the other hand, at each $t^n$ we have \eqref{Vnu-tn}.
Therefore, if we set
\begin{equation*}
 a_n = |{\mathcal V}^\nu(t^{n}+)| 
\end{equation*}
use \eqref{eq:u-v_fractional-step} for $\DT<1/M$, then
\begin{align*}
a_n 
&\le \left(1-M\DT\right) \left(a_{n-1} + \eta_\nu C \DT \right)
\end{align*}
with $C$ as above. We find that for all $n\ge 1$,
\begin{align}\label{eq:estimate_an}
 a_n &\le \left(1-M\DT\right)^n a_0 +  \frac{C} {M} \, \eta_\nu 
    \le \left(1-M\DT\right)^n \frac{M}{\nu} +  \frac{C} {M} \, \eta_\nu
\end{align}
where we used \eqref{int-v0-to-0}. Hence the estimate \eqref{estimate:mathcalV} holds for $t=t^n+$. 
Let us consider $t\in(t^n,t^{n+1}),$ then we can write 
\begin{align*}
    \modulo{\mathcal{V}^\nu(t)-\mathcal{V}^\nu(t^n+)}\leq C \eta_\nu (t-t^n).
\end{align*}
Using \eqref{eq:estimate_an}, we conclude that \eqref{estimate:mathcalV} holds for every $t$.

\smallskip
Next, we show estimate~\eqref{estimate:mathcalU}. Recalling \eqref{eq:u-v_fractional-step}$_1$, the map 
$t\mapsto{\mathcal U}^\nu(t) = \int_0^M u^\nu(y,t)\,dy$ is continuous and piecewise linear, with derivative possibly discontinuous 
at the interactions times and at the time steps. Now we estimate its derivative, for $t$ not being interaction point or time step: 
$$\frac{d}{dt}\mathcal{U}^\nu(t)=-\sum_{j=0}^{N^\nu(t)+1} \Delta u^\nu (y^\nu_j)(y^\nu_j)'= (I)+(II)\;,$$
where 
$$
(I) = -\sum_{j=0}^{N^\nu(t)+1} \Delta u^\nu (y^\nu_j) \mu^\nu_j \;, \quad (II) =- \sum_{j=0}^{N^\nu(t)+1} \Delta u^\nu (y^\nu_j) \left[ (y^\nu_j)' - \mu^\nu_j \right]\,,
$$
with $\mu^\nu_j$ given by the exact Rankine-Hugoniot speed.
From the Rankine-Hugoniot condition (RH)  for \eqref{eq:system_Lagrangian_M}, we have
\begin{align*}
    \Delta u^\nu (y^\nu_j) \mu^\nu_j =\Delta (v^\nu)(y^\nu_j), \quad j=1,...,N^\nu(t)\,,
\end{align*}
hence $(I)$ can be written as a telescopic sum, and then $(I)=0$ by periodicity.

In addition, from Lemma \ref{lem:sup-inf-u}, we obtain
\begin{align*}
|(II)| &\le \eta_\nu \tv(u^\nu(t))\\
&\leq  \eta_\nu \tv(\ln(u^\nu(t))) \sup(u^\nu(t))\\
&\leq \tilde C \eta_\nu \inf(u^\nu((t))\,,
\end{align*}
with $\tilde C\doteq  2q  e^{2 q}$ independent of $\nu$. Combining the above, integrating over $[0,t]$, and recalling \eqref{int-u0-nu}$_2$, we arrive at
\begin{align*}
   \left| \mathcal{U}^\nu(t)-\ell\right|\leq \frac{M}{\nu} ~+~ \tilde C  \, \eta_\nu\,\int_0^t\inf(u^\nu(\tau))\,d\tau,
\end{align*}
which yields \eqref{estimate:mathcalU}. This completes the proof.
\end{proof}

\begin{lemma}\label{Lemmauinfsup}
Having the initial data~\eqref{eq:init-data-lagr}, the approximate sequence $(u^\nu, v^\nu)$ to~\eqref{eq:system_Lagrangian_M} 
is defined for all times $t>0$, and the following holds.
\begin{itemize}
    \item[(i)] There exist positive constants $C_1,\,C_2,\,C_3$ and $C_4$, independent on $\nu$ and $t$, such that 
\begin{align}\label{eq:ast}
    C_3 e^{-C_4 \eta_\nu t   -2q } &\le  \inf \{u^\nu (\cdot,t)\}\le  \sup \{u^\nu (\cdot,t) \}\leq C_1 e^{C_2 \eta_\nu t { +2q }} 
    \end{align}
and there exist constant values $0<u_{inf}\le u_{sup}$ such that 
\begin{align}\label{eq:unif-bounds}
      0< u_{inf} 
      &\le  u^\nu (y,t)\le u_{sup}\qquad \forall \,
      (y,t)\in T_M\times [0,1/\sqrt{\eta_\nu})  \,.
    \end{align}

\item[(ii)] There exists a constant $\widetilde C_0>0$ such that
\begin{equation}\label{eq:unif-bounds-v_nu}
    |v^\nu(y,t)|\le \widetilde C_0\qquad \forall\, y\,,\ t\,,\ \nu\,. 
\end{equation}
\end{itemize}
\end{lemma}

\begin{proof} 
Observing that
\begin{equation*}
  M \, \inf \{ u^\nu(\cdot,t) \} \, \le\, \mathcal{U}^\nu(t) \,\le\, M\, \sup \{u^\nu(\cdot,t)\}
\end{equation*}
and calling $z(t)\doteq \inf \{u^\nu(\cdot,t)\}$, we find that
\begin{equation*}
0\le z(t) \le C_1 + C_2  \, \eta_\nu\,\int_0^t z(\tau)\,d\tau\,,\qquad C_1>{\frac\ell M} + \frac{1}{\nu}\,,\quad C_2=  \frac {{\widetilde{C}} }{M}, \,
\end{equation*}
using Lemma~\ref{L4} with $C_1,\,C_2$ independent of $t$ and $\nu$.
By Gronwall's inequality, we find that
$$
z(t)\le C_1 e^{C_2 \eta_\nu t }\,.
$$
For $T>0$  and $t\in[0,T]$, the relation~\eqref{sup-inf-u} implies that $\inf \{u^\nu(\cdot,t)\}$, $\sup \{u^\nu(\cdot,t)\}$ 
are uniformly bounded from above on $[0,T]$ and for every $\nu$, since $\eta_\nu \to 0$ as $\nu\to\infty$.

Similarly, from the inequality
\begin{align*}
   \left| \mathcal{U}^\nu(t)-\ell\right|\leq \frac{M}{\nu} ~+~ {\widetilde{C}}  \, \eta_\nu\,\int_0^t\sup \{u^\nu(\cdot,\tau)\}\,d\tau
\end{align*}
and setting $w(t)\doteq \sup \{u^\nu(\cdot,t)\}$, we deduce that
\begin{equation*}
  \ell - \frac{M}{\nu} ~-~ {\widetilde{C}}   \, \eta_\nu\,\int_0^t w(\tau)\,d\tau \le M w(t).
\end{equation*}
Then, 
\begin{equation*}
  C_3 ~-~ C_4  \, \eta_\nu\,\int_0^t w(\tau)\,d\tau \le w(t),\quad C_3<\frac{\ell}{M} - \frac{1}{\nu}, \quad C_4=\frac{\widetilde{C}}{M},
\end{equation*}
with $C_3$ and $C_4$ independent of $t$ and $\nu.$
By Gronwall inequality we find that
\begin{equation*}
    w(t)\geq C_3 e^{-C_4 \eta_\nu t}.
\end{equation*} 
For $t\in[0,T]$ with $T>0$ fixed, we find that $\inf u^\nu(t)$, $\sup u^\nu(t)$ are uniformly bounded from below as well on $[0,T]$ and for every $\nu$.
Applying the result in \eqref{sup-inf-u}, we get \eqref{eq:ast}.

Taking now $t_\nu=\frac{1}{\sqrt{\eta_\nu}} $ in \eqref{eq:ast}, we arrive at \eqref{eq:unif-bounds} immediately for suitable constants $0<u_{inf}\le u_{sup}$ which are independent on $\nu$, { but depend on $q$}.

\smallskip
To prove the bound on $v^\nu$, we proceed in a similar fashion. Using \eqref{estimate:mathcalV}, we can write  
\begin{align*}
M \inf \{v^\nu(\cdot,t)\}\leq \mathcal{V}^\nu(t) \leq M \sup \{v^\nu(\cdot,t)\}.
\end{align*}
Employing estimate~\eqref{estimate:mathcalV}, we have
\begin{align*}
\inf \{v^\nu(\cdot,t)\}\leq e^{-Mt} e^{M \Delta t_\nu}\frac{1}{\nu}+\frac{C}{M^2} \eta_\nu+\frac{C}{M}\eta_\nu \Delta t_\nu.
\end{align*}
We also have
\begin{align*}
    \sup  \{v^\nu(\cdot,t) \}- \inf \{ v^\nu(\cdot,t) \}\le \tv  v^\nu(\cdot,t) \le C_1
\end{align*}
from  \eqref{eq:unif-bound-tv-v}, and this implies 
\begin{align*}
    \sup \{ v^\nu(\cdot,t) \} &\le C_1+ \inf  \{v^\nu(\cdot,t)\}\leq C_1+ \frac{1}{M}\mathcal{V^\nu}(t)\\
   & \le C_1+ e^{-Mt} e^{M \Delta t_\nu}\frac{1}{\nu}+\frac{C}{M^2} \eta_\nu+\frac{C}{M}\eta_\nu \Delta t_\nu.
\end{align*}
On the other hand, we have
\begin{align*}
C_1+\inf \{v^\nu(\cdot,t)\}& \geq \sup \{v^\nu(\cdot,t)\}\geq \frac{1}{M}\mathcal{V^\nu}(t)\\
&\geq -e^{-Mt} e^{M \Delta t_\nu}\frac{1}{\nu}-\frac{C}{M^2} \eta_\nu-\frac{C}{M}\eta_\nu \Delta t_\nu,
\end{align*}
so that
\begin{align*}
\inf \{v^\nu(\cdot,t)\}\geq -e^{-Mt} e^{M \Delta t_\nu}\frac{1}{\nu}-\frac{C}{M^2} \eta_\nu-\frac{C}{M}\eta_\nu \Delta t_\nu -C_1.
\end{align*}
These inequalities imply \eqref{eq:unif-bounds-v_nu} for a suitable constant $\tilde C_0,$ independent of $\nu,\,t.$

It remains now to show that the approximate sequence is defined for all times $t>0$. To achieve this, we need to verify that the size of rarefaction fronts remains less than $\eta$ for all $t>0$ and the interaction points do not accumulate. Indeed, by construction, the size of each rarefaction is bounded by $\eta_\nu$ at time $t=0+$, and after an interaction with other waves, the size of a rarefaction does not increase. However, after a time step, newly generated rarefactions that may result as a reflected wave of a shock have strength estimated 
by~\eqref{time-step-reflected-wave}:
$$
|\eps_1^+|\le C_1^-(q) {M}\DT |\eps_2^-| \,.
$$
Therefore, if we choose $\DT=\DT_\nu$, $\eta=\eta_\nu$ such that
\begin{equation}
     C_1^-(q) {M}\DT \, q  \le \eta\,,
\end{equation}
then every newly generated rarefaction is guaranteed to have size $\le \eta$, since $|\eps_2^-|\le q$. 
As a consequence, the rarefaction would not be divided into two or more, according to the algorithm in Subsection~\ref{S2.3}. 

Furthermore, we can see from~\eqref{eq:charact_speed} and \eqref{eq:unif-bounds} 
that the propagation speeds of the waves are uniformly bounded and the ranges 
of values of each family are separated. Actually, because of periodicity, waves reaching one boundary enter from the other.
Thus, we can apply the analysis in~\cite[Sect. 3]{AG_2001}, and prove that, except for a finite number of interactions, 
there is at most one outgoing wave of each family for each interaction. In view of this analysis, the approximate sequence $(u^\nu, v^\nu)$ 
is defined for all times $t>0$ and hence the bounds in  \eqref{eq:unif-bounds} are valid.
The proof is now complete.
\end{proof}

\subsection{A weighted functional} \label{subsect:Lxi}
We bring now into play a weighted functional introduced in ~\cite{AG_2001} that helps us measure the wave cancellation between time steps in a better way than the linear functional $L(t)$. This is the weighted functional
\begin{equation}\label{def:Lxi}
L_\xi(t) = \sum_{j=1,\ \eps_j>0}^{N(t)} |\eps_j| + \xi \sum_{j=1,\ \eps_j<0}^{N(t)} |\eps_j|\qquad \xi\ge1\,,
\end{equation}
that is all shocks are multiplied by a weight $\xi\ge 1$. Notice that  for $\xi=1$, the functional coincides with $L(t)$, i.e. $L_1=L$. Also, it is clear that  $L(t)\le L_\xi(t) \le \xi L(t)$ for all $t$, hence $L_\xi(t)$ is defined for all times and we have,
from Lemmas~\ref{lem:bounds-on-bv} and~\ref{Lemmauinfsup}, that
\begin{equation}\label{bound-on-Lxi}
    L_\xi(t) \le \xi L(0+)\qquad \forall \, t>0\,.
\end{equation}

Now, following the analysis in~\cite{ABCD_JEE_2015}, we bound the decrease of $L_\xi(t)$ for all times different from time steps. More precisely,  if 
\begin{equation}\label{bounds-on-xi}
1\le \xi\le \frac 1 {c(q)}\,, \qquad  c(q) := \frac{\cosh(q)-1}{\cosh(q)+1}
\end{equation}
and $L(0+)\le q$, then $\Delta L_\xi(t)\le 0$ for any $t>0$, $t\not = t^n$ for $n\ge 1$. 
In particular, if $t\ne t^n$ is an interaction time between two waves of the same family, and with the notation of (a) in Proposition~\ref{S2:prop:estimate-time-step},
it holds
\begin{equation}\label{eq:refined-decay-Lxi}
 \Delta L_\xi(t) + \left(\xi-1\right)|\eps_{refl}| \le 0\,.
\end{equation}
In particular, for $\xi=1$ \eqref{eq:refined-decay-Lxi} reduces to $\Delta L(t)\le 0$, see~\eqref{eq:Delta-L}. 
The proof of~\eqref{eq:refined-decay-Lxi} can also be found in~\cite[App. A.3]{AC2022}.

On the other hand, the functional $L_\xi$ may increase across the time steps, for $\xi>1$: from \cite[(3.30)]{AC2022}, 
\begin{equation}\label{eq:Delta-L-xi-time-step}
\Delta L_\xi(t^n) \le \frac{M}2  \DT \, (\xi-1) L(t^n-)\,.
\end{equation}

\subsection{Estimates on vertical traces}\label{subsec:vert-traces}
Here, we study the total variation of the approximate solutions in time for fixed  $Y\in[0,M)$. Given $Y\in [0,M)$ and $t>0$, we define the functional
\begin{equation*}
    W^\nu_Y(t) = \frac 12 \tv\{ \ln (u^\nu)(Y,\cdot); (0,t)\}\,.
\end{equation*}

By the definition of the strengths, $  W^\nu_Y(t)$ is the sum of the strengths $|\eps|$ of the waves that cross $Y$ 
within the time interval $(0,t)$ and the goal is to find a uniform bound for $W^\nu_Y$ on bounded intervals of time, 
independently of $Y\in [0,M)$. 

Given $(Y,T)$, we set the triangular domain 
\begin{equation}\label{triangleD}
D=\left\{ (y,t):\,\, y\in I(t), t\in [0,T] \right\}
\end{equation}
with $I(t)=[Y-\lambda^\ast (T-t),\, Y+\lambda^\ast (T-t) ]$ where $\lambda^\ast=\frac{\alpha}{u_{inf}}$. 
The slope of the sides of the triangle is selected so that no wave present outside this triangular region can enter at a later time. 
In other words, $D$ depends on $(Y,T)$ and is the triangle with top vertex $(Y,T)$ and base $(Y-\lambda^\ast T,\, Y+\lambda^\ast T)\times \{t=0\}.$

Next, we define the weighted functional
\begin{equation}\label{def:LxiD}
L_{\xi,D}(t) = \sum_{j\in I(t),\ \eps_j>0} |\eps_j| + \xi \sum_{j\in I (t),\ \eps_j<0} |\eps_j|,\qquad \xi\ge1\,, \quad t\in [0,T]
\end{equation} 
in which we consider only all the waves {present} in the triangle $D$ at time $t$.
In the case of $\xi=1$, the functional coincides with $L_D(t),$ i.e.
\begin{equation}\label{def:LD}
L_{D}(t) = \sum_{j\in I(t),\ \eps_j>0} |\eps_j| + \sum_{j\in I(t),\ \eps_j<0} |\eps_j|, \quad \quad t\in [0,T].
\end{equation} 
\begin{lemma} There exist $\widetilde C_1( q)$,   $\widetilde C_2(q)$ such that, for every  $Y\in[0,M),\,\nu\in \N$ and $T>0$, 
the following holds:
\begin{align}\label{eq:bound-on-TV_at_y}
    W^\nu_Y(T)\le  \widetilde C_1 (q) \, \zeta(0) +  M \widetilde C_2(q)  \int_0^{T} L(t) \,\zeta(t)\,dt \,. 
\end{align} 
with $\zeta(t):= \Big\lfloor\frac{2\lambda^*(T-t)}{M}\Big \rfloor +1,$ where $\lfloor \cdot \rfloor$ denotes the integer part.
\end{lemma}
\begin{proof}
Fix  $Y\in [0,M)$, $T>0$ and take  $A_Y(t)$ to be the sum of the sizes of the waves that, at time $t$, are approaching the position $Y$  within the triangle $D$: 
{ $$A_Y(t)=\displaystyle{\sum_{\text{approaching $Y$ and $\varepsilon\in I (t)$}}\modulo{\varepsilon}}\,.$$}\\
It consists of the waves of the $2^{nd}$ family, located to the left of  $Y$  and the waves of the $1^{st}$ family, located to the right of $Y$.

Now, we claim that the map
\begin{equation*}
    t\mapsto W^\nu_Y(t) + A_Y(t) +  \kappa   L_{\xi,D}(t)\,, \quad \kappa=\frac1{\xi-1}
\end{equation*}
is nonincreasing at times $t\not = t^n$, i.e.
\begin{equation}\label{Delta-W-one}
\Delta\left(W^\nu_Y + A_Y + \kappa L_{\xi,D}\right)(t) \le 0;    
\end{equation}
while at time steps, it holds
\begin{equation}\label{Delta-W-two}
\Delta\left(W^\nu_Y + A_Y + \kappa  L_{\xi,D}\right)(t^n) \le \DT\,  L_D(t^n) \, M\,\widetilde C_2\,,\qquad 
\widetilde C_2 = \,C_1^-(q) + \frac{ 1}{2} ,\,
\end{equation}

\begin{figure}
	\centering
    \begin{tikzpicture}[scale=1]
         \draw[line width=1] (-2.5,0) -- (9,0); 
         \draw[line width=1] (8.8,0.2) -- (9,0);   
                  \draw[line width=1] (8.8,-0.2) -- (9,0);        
           \draw[line width=1] (2.4,-0.3) -- (2.4,6); 
           \draw[line width=1] (2.2,5.8) -- (2.4,6); 
           \draw[line width=1] (2.6,5.8) -- (2.4,6);   
                  \draw[line width=1,-, dotted] (2.8,-0.3) -- (2.8,5); 
       \draw[line width=1,dashed] (4.,-0.3) -- (4.,5);
             \draw[line width=1,dashed] (5.6,-0.3) -- (5.6,5);
                \draw[line width=1,dashed] (7.2,-0.3) -- (7.2,5);
        \draw[line width=1,dashed] (0.8,-0.3) -- (0.8,5); 
        \draw[line width=1,dashed] (-0.8,-0.3) -- (-0.8,5); 
           \draw[line width=1,dashed] (-2.4,-0.3) -- (-2.4,5);

         \draw (2.4,-0.6) node {$0$};   
          \draw (2.8,-0.6) node {$Y$}; %
        \draw (4.,-0.6) node {$M$}; 
        \draw (0.8,-0.6) node {$-M$}; 
         \draw (-1,-0.6) node {$-2M$};   
       
        \draw (5.6,-0.6) node {$2M$}; %
        \draw (2.8,5.3) node {$\,\,\quad (Y,T)$}; %
        \draw (3.1, 4) node {$D$}; %
        \draw (9,-0.4) node {$y$}; %
               \draw (2.,5.8) node {$t$}; %
     \draw[line width=1,-,blue] (2.8,5) -- (-2.6,0.);
      \draw[line width=1,-,blue] (2.8,5) -- (8.2,0.);
        \draw[line width=2,-] (2.8,-0.1) -- (2.8,0);
        \draw[line width=2,-] (4.4,-0.1) -- (4.4,0);
        \draw[line width=2,-] (6.0,-0.1) -- (6.0,0);     
                \draw[line width=2,-] (7.6,-0.1) -- (7.6,0);     
             \draw[line width=2,-] (1.2,-0.1) -- (1.2,0);
        \draw[line width=2,-] (-0.4,-0.1) -- (-0.4,0);
        \draw[line width=2,-] (-2.0,-0.1) -- (-2.0,0);    
        
       {
       \draw[line width=1,-,  blue] (0.5, 1.3) -- (0.27, 0.2); 
              \draw[line width=1,-,  blue] (2.1, 1.3) -- (1.87, 0.2);      
      \draw[line width=1,-,  red] (3.7, 1.3) -- (3.47, 0.2); 
            \draw[line width=1,-,  red] (5.3, 1.3) -- (5.07, 0.2); 
                                   }
      \draw[line width=1,-, blue] (0.5, 1.3) -- (0., .1);
     \draw[line width=1,-, blue] (2.1, 1.3) -- (1.6, .1);
        \draw[line width=1,-, red] (3.7, 1.3) -- (3.2, .1);
        \draw[line width=1,-, red] (5.3, 1.3) -- (4.8, .1);
                                    
      \draw[line width=1,-,  blue] (0.5, 1.3) -- (0.8, 1.9);
            \draw[line width=1,-,blue] (2.1, 1.3) -- (2.4, 1.9);
     \draw[line width=1,-, red] (3.7, 1.3) -- (4., 1.9);
            \draw[line width=1,-, red] (5.3, 1.3) -- (5.6, 1.9);
      \draw[line width=1,-, red] (0.5, 1.3) -- (0.1, 2.1);
      
      \draw[line width=1,-,red] (2.1, 1.3) -- (1.7, 2.1);

       \draw[line width=1,-, blue] (3.7, 1.3) -- (3.3, 2.1);

       \draw[line width=1,-, blue] (5.3, 1.3) -- (4.9, 2.1);

    \end{tikzpicture}
    \caption{The interaction of waves in case \textbf{c} within the triangle $D$ defined in \eqref{triangleD}. The difference in colours indicates approaching (blue) and not approaching (red) wave fronts.} \label{fig:case_c}
  \end{figure}
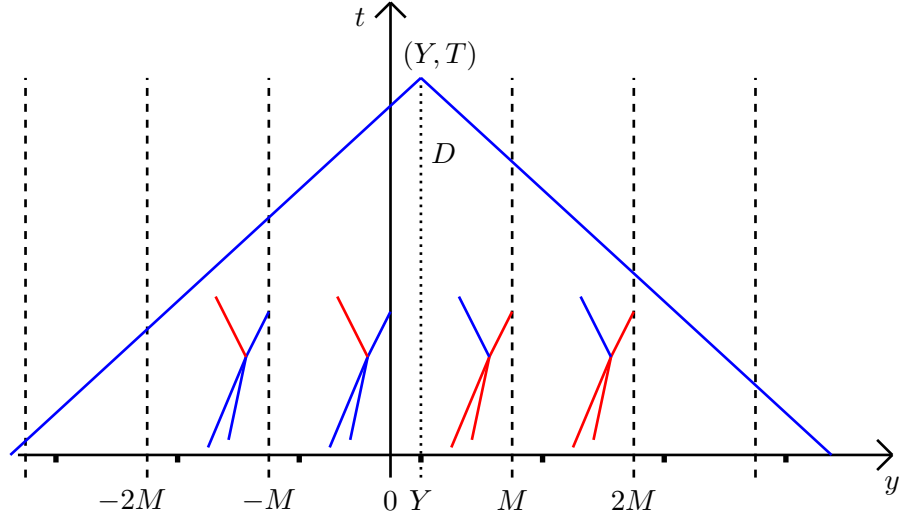


\begin{figure}
	\centering
    \begin{tikzpicture}[scale=1]
         \draw[line width=1] (-2.5,0) -- (9,0); 
         \draw[line width=1] (8.8,0.2) -- (9,0);   
                  \draw[line width=1] (8.8,-0.2) -- (9,0);        
           \draw[line width=1] (2.4,-0.3) -- (2.4,6); 
           \draw[line width=1] (2.2,5.8) -- (2.4,6); 
           \draw[line width=1] (2.6,5.8) -- (2.4,6);   
                  \draw[line width=1,-, dotted] (2.8,-0.3) -- (2.8,5); 
       \draw[line width=1,dashed] (4.,-0.3) -- (4.,5);
             \draw[line width=1,dashed] (5.6,-0.3) -- (5.6,5);
                \draw[line width=1,dashed] (7.2,-0.3) -- (7.2,5);
        \draw[line width=1,dashed] (0.8,-0.3) -- (0.8,5); 
        \draw[line width=1,dashed] (-0.8,-0.3) -- (-0.8,5); 
           \draw[line width=1,dashed] (-2.4,-0.3) -- (-2.4,5);

         \draw (2.4,-0.6) node {$0$};  
         \draw (2.8,-0.6) node {$Y$}; %
        \draw (4.,-0.6) node {$M$}; 
        \draw (0.8,-0.6) node {$-M$}; 
         \draw (-1,-0.6) node {$-2M$};   
       
        \draw (5.6,-0.6) node {$2M$}; %
        \draw (2.8,5.3) node {$\,\,\quad (Y,T)$}; %
        \draw (3.1, 4) node {$D$}; %
        \draw (9,-0.4) node {$y$}; %
               \draw (2.,5.8) node {$t$}; %
                              
     \draw[line width=1,-,blue] (2.8,5) -- (-2.6,0.);
      \draw[line width=1,-,blue] (2.8,5) -- (8.2,0.);
        \draw[line width=2,-] (2.8,-0.1) -- (2.8,0);
        \draw[line width=2,-] (4.4,-0.1) -- (4.4,0);
        \draw[line width=2,-] (6.0,-0.1) -- (6.0,0);     
                \draw[line width=2,-] (7.6,-0.1) -- (7.6,0);     
             \draw[line width=2,-] (1.2,-0.1) -- (1.2,0);
        \draw[line width=2,-] (-0.4,-0.1) -- (-0.4,0);
        \draw[line width=2,-] (-2.0,-0.1) -- (-2.0,0);    
        
       {
       \draw[line width=1,-,  blue] (-0.4, 1.3) -- (-0.63, 0.2); 
          \draw[line width=1,-,  blue] (1.2, 1.3) -- (0.97, 0.2); 
          \draw[line width=1,-,  blue] (2.8, 1.3) -- (2.57, 0.2); 
            \draw[line width=1,-,  red] (4.4, 1.3) -- (4.17, 0.2); 
            \draw[line width=1,-,  red] (6, 1.3) -- (5.77, 0.2);

                                   }
      \draw[line width=1,-, blue] (-0.4, 1.3) -- (-0.7, .5);
         \draw[line width=1,-, blue] (1.2, 1.3) -- (0.9, .5);
        \draw[line width=1,-, blue] (2.8, 1.3) -- (2.5, .5);
        \draw[line width=1,-, red] (4.4, 1.3) -- (4.1, .5);
   \draw[line width=1,-, red] (6, 1.3) -- (5.7, .5);

      \draw[line width=1,-,  blue] (-0.4, 1.3) -- (0, 2);
      \draw[line width=1,-,  blue] (1.2, 1.3) -- (1.6, 2);
      \draw[line width=1,-,  red] (2.8, 1.3) -- (3.2, 2);
      \draw[line width=1,-,  red] (4.4, 1.3) -- (4.8, 2);
      \draw[line width=1,-,  red] (6, 1.3) -- (6.4, 2);
   
      \draw[line width=1,-, red] (-0.4, 1.3) -- (-0.7, 2.1);
      \draw[line width=1,-, red] (1.2, 1.3) -- (0.9, 2.1);
      \draw[line width=1,-, red] (2.8, 1.3) -- (2.5, 2.1);
      \draw[line width=1,-, blue] (4.4, 1.3) -- (4.1, 2.1);
      \draw[line width=1,-, blue] (6, 1.3) -- (5.7, 2.1);

    \end{tikzpicture}
    \caption{The interaction of waves in case \textbf{d} within the triangle $D$ defined in \eqref{triangleD}. The difference in colours indicates approaching (blue) and not approaching (red) wave fronts.} \label{fig:case_d}
  \end{figure}
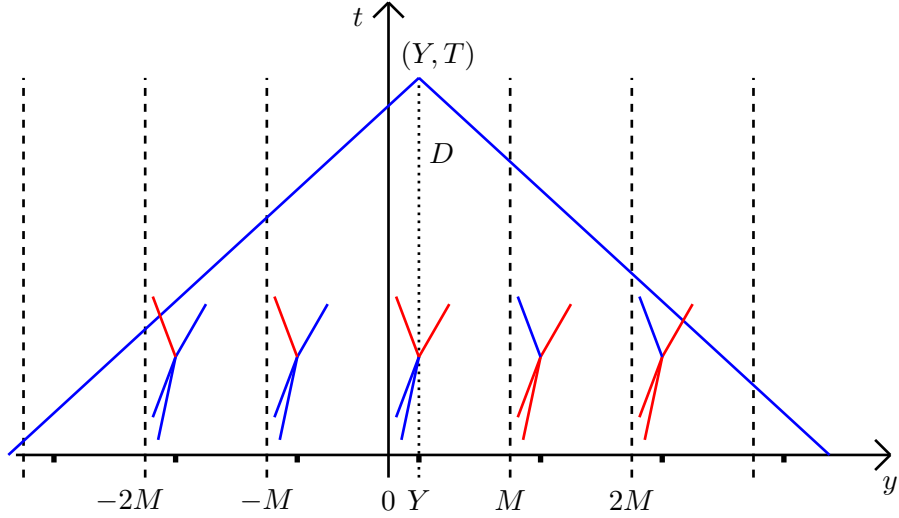
\noindent
with $C^-_{1}(q) = \cosh(q)/2$, from \eqref{time-step-reflected-wave}. 

To prove the claim, we introduce the following integers:\\
$\bullet$\quad $\gamma_A$ is the number of times that a wave $\eps$ is approaching $Y$ due to periodicity,  \\
$\bullet$\quad $\gamma_L$ is the number of times that a wave $\eps$ is present at time $t$ in $I(t)$ 
due to periodicity.  Then, for a wave in $D$ that is approaching $Y,$ we immediately get the following relation 
\begin{equation}\label{gamma A and L} 1\le \gamma_A\le \gamma_L\le \Big\lfloor\frac{2\lambda^* (T-t)}{M}\Big \rfloor +1.\,\end{equation}

Notice that $t\mapsto W^\nu_Y(t) + A_Y(t) + \kappa  L_{\xi,D}(t)$ is piecewise constant 
and can change value only in the following cases:

\smallskip{}
{\bf a.}\quad If a wave front of size $\eps$ crosses the position $Y$, then \eqref{Delta-W-one} holds because
 $$
\Delta W^\nu_Y(t)=|\eps|\,,\quad  \Delta A_Y(t) = - |\eps|\,,\quad  \Delta L_{\xi,D}(t)=0;
$$

\smallskip{} 
{\bf b.}\quad If a wave front of size $\eps$ reaches the left or the right side of the triangle $D$, then
$$
\Delta W^\nu_Y(t)=0= \Delta A_Y(t)\,,\quad  \Delta L_{\xi,D}(t)\leq -\modulo{\varepsilon};
$$
and~\eqref{Delta-W-one} holds.

\smallskip{}
{\bf c.}\quad If $t$ is an interaction time at which two wave fronts occur at a point $y\not = Y$, then $\Delta W^\nu_Y(t)=0$. 
 Now, regarding the interacting waves, we have    
 $$
\Delta A_Y(t) =\gamma_A\cdot  \left(|\eps|-|\alpha|-|\beta|\right) + (\gamma_L-\gamma_A)  |\eps_{refl}| = (\gamma_L-2\gamma_A)  |\eps_{refl}| $$
while
$$\Delta L_{\xi, D}(t)\le - \gamma_L\cdot  (\xi-1)|\eps_{refl}| \;,
$$
adopting  the notation of (a) in Proposition~\ref{S2:prop:estimate-time-step} and using \eqref{eq:refined-decay-Lxi}.

More precisely, the term $\gamma_A\cdot  \left(|\eps|-|\alpha|-|\beta|\right)$ takes into account the case of interaction points 
with incoming fronts approaching $Y$ and hence the transmitted one is also approaching, while the reflected front is not approaching. 
On the other hand, the term  $(\gamma_L-\gamma_A)  |\eps_{refl}|$ involves the interaction cases with incoming fronts not approaching of $Y$, 
hence, having only the reflected front as approaching. See Fig. \ref{fig:case_c}.
Hence,
\begin{equation*}
    \Delta (W^\nu_Y+A_Y+\kappa L_{\xi, D})\le \modulo{\eps_{refl}}\left(\gamma_L-2\gamma_A- \kappa \gamma_L(\xi-1)\right) \leq 0, 
\end{equation*}
since $\kappa=\frac{1}{\xi-1}>0$ and using~\eqref{gamma A and L}. Thus, \eqref{Delta-W-one} holds true.

\smallskip{}
 {\bf d.}\quad  Finally, if the interaction occurs at the point  $Y$, then
 $$
\Delta W^\nu_Y(t)= |\eps|\,,\quad\quad  \Delta L_{\xi,D}(t)\le  -\gamma_L (\xi-1)|\eps_{refl}|\le 0\,,
$$
and
\begin{align*}
\Delta A_Y(t)=&-\gamma_A\cdot\left( |\alpha| + |\beta|\right) +|\eps|(\gamma_A-1)+(\gamma_L-\gamma_A)\,|\eps_{refl}|\\
&=\gamma_A\cdot\left(|\eps|- |\alpha| - |\beta|\right) -|\eps|+(\gamma_L-\gamma_A)\,|\eps_{refl}|\\
&= (\gamma_L-2\gamma_A)\,|\eps_{refl}| -|\eps|
\end{align*}
where we used \eqref{eq:Delta-L}. 
Then,
$$
\Delta W^\nu_Y(t) + \Delta A_Y(t)+\kappa \Delta L_{\xi,D}(t) \le    -2\gamma_A\,|\eps_{refl}|  \le 0 
$$
and hence \eqref{Delta-W-one} holds in all cases. See Figure \ref{fig:case_d}.

In the case that $t=t^n$,  we show that
{ \begin{equation}\label{eq:DeltaW-case_d}
    \Delta W^\nu_Y(t^n) + \Delta A_Y(t^n)\le \DT\, C_1^-(q)M L_D(t^n)  \,.
\end{equation}
}
Indeed, if no wave reaches the position  $Y$ at time $t^n$, then 
 $$
\Delta W^\nu_Y(t^n)=0  \,,\quad  \Delta A_Y(t^n)\le \DT\, C_1^-(q)M L_D(t^n)  
$$
with $C_1^-(q) = \cosh(q)/2$. 
On the other hand, if a wave of size $\eps^-$ reaches  $Y$ at time $t^n$, then
$$
\Delta W^\nu_Y(t^n)=|\eps^+|  \,,\quad  \Delta A_Y(t^n)\le -|\eps^-| \cdot\gamma_A + |\eps^+|(\gamma_A -1)+ \DT\, C_1^-(q)M \, L_D(t^n)  ,
$$
where $\eps^+$ is the size of the transmitted wave. From \eqref{Delta-L-time-step}, 
one has that $|\eps^+|< |\eps^-|$ and hence \eqref{eq:DeltaW-case_d} holds also in this case. Thus \eqref{eq:DeltaW-case_d} is true.

We combine now \eqref{eq:DeltaW-case_d} and \eqref{eq:Delta-L-xi-time-step} to get
\begin{equation*}
\Delta\left(W^\nu_Y + A_Y + \kappa L_{\xi,D}\right)(t^n) \le \DT\,L_D(t^n) \, M \left(C_1^-(q)  + \frac{1}{2}  \right)
\end{equation*}
and hence \eqref{Delta-W-two} holds with  $\widetilde C_2 (q)\, \dot = \,C_1^-(q)+ \frac{1}{2 }$\,.

Having established the claim,  we apply \eqref{Delta-W-one} and \eqref{Delta-W-two} to find
\begin{align}
    W^\nu_Y(T)&\le  W^\nu_Y(T) + A_Y(T) + \kappa L_{\xi, D}(T) \nonumber\\
    &= \underbrace{W^\nu_Y(0)}_{=0} + \underbrace{A_Y(0)}_{\le L_D(0)} 
    + \kappa \underbrace{L_{\xi, D}(0)}_{\le \xi  L_D(0)} + \sum_{0<t<T}\Delta\left(W^\nu_Y + A_Y + \kappa L_{\xi, D}\right)(t)  \nonumber\\
    &\le \left(1+ \frac{\xi }{\xi-1}\right)L_D(0)  + \sum_{j=1}^n \Delta\left(W^\nu_Y + A_Y + \kappa L_{\xi, D}\right)(t^j)  \nonumber\\
    &\le  \left(2+ \frac{1}{\xi-1}\right)L_D(0) +  M \widetilde C_2(q)\int_0^{t^n} L_D(t) \,dt \label{eq:4}
\end{align}
for the integer $n$ satisfying $t^n < T \le t^{n+1}$. Note that we used that $L_D(t)$ is non-increasing on $(t^{n-1},t^{n})$ and hence,
$$
\DT\,L_D(t^n) \le \int_{t^{n-1}}^{t^n} L_D(t) \,dt\,.
$$
Now, choose $\xi=c(q)^{-1} = (\cosh(q) +1)/(\cosh(q) -1)$  in \eqref{bounds-on-xi}, and set
$$
\widetilde C_1 := \left( 2 + \frac{1}{\xi-1}\right) q= \left(\frac{3+\cosh(q)}{2}\right) q\, \ge q\,.$$
By the definition of $I(t)$, $L_D(t)$ can be estimated as follows:
$$L_D(t)\le L(t)\left( \Big\lfloor\frac{2\lambda^*(T-t)}{M}\Big \rfloor +1\right), \qquad 0\le t\le T\,.$$
These last estimates, together with \eqref{eq:4}, imply that \eqref{eq:bound-on-TV_at_y} holds for  $0\le Y<M$\,.
\end{proof}

Now, following the proof of~\cite[Lemma 3.7]{AC2022}, the estimate \eqref{eq:bound-on-TV_at_y} yields that, for every $T>0$, 
there exist positive constants $C$ and  $L$ independent of $\nu\in\N$ such that for all $y\in \mathbb{T}_M$, 
\begin{equation}\label{bound-bv-time}
    \tv\{\left(u^\nu,v^\nu\right)(y,\cdot);[0,T]\} \le C 
\end{equation}
and $L^1$-stability on bounded intervals of time holds:
\begin{align}\label{bound-integral-time}
   \int_0^T \left| v^\nu(y_1,t) -  v^\nu(y_2,t)\right|\, dt \le L \left| y_2 - y_1  \right|\qquad \forall\, y_1\,,\, y_2\in \mathbb{T}_M\,.
\end{align}
Stability estimate~\eqref{bound-integral-time} holds also for $u^\nu$.

\subsection{Proof of Theorem 2}\label{S2.8}
In this subsection we show a compactness property of the sequence of approximate solutions $\left(u^{\nu}, v^{\nu}\right)$, 
constructed in Subsection~\ref{S2.3}, as $\nu\to\infty$ and $\DT=\DT_\nu\to 0$. In particular, up to a subsequence, 
we find in the limit an entropy weak solution of the problem \eqref{eq:system_Lagrangian_M}-\eqref{eq:init-data-uv}, 
in the sense of Definition \ref{entropy-sol_lagr}.

\begin{theorem}\label{Th-1-lagr} Let $(u_0,v_0)$ satisfy \eqref{eq:init-data-lagr}.
Then there exists a subsequence of $(u^{\nu}, v^{\nu}),$ still denoted by $(u^{\nu}, v^{\nu})$ 
which converges, as $\nu\to\infty$, to a function $(u,v)$ in $L^1_{loc}\left(\mathbb{T}_M\times[0,+\infty) \right)$, and the following holds:
\begin{itemize}
\item [a)] the map $t\mapsto (u,v)(\cdot,t)\in L^1(\mathbb{T}_M)$ is well-defined and it is Lipschitz continuous 
in the $L^1$--norm. It satisfies $(u,v)(\cdot,0)=(u_0,v_0)$ and for all $t>0$
\begin{equation}\label{eq:Linfty-bound-uv}
0<u_{inf}\le u(y,t) \le u_{sup}\,,\quad      |v(y,t)|\le \widetilde C_0   \qquad \mbox{for a.e.}\ y
\end{equation}   
with $u_{inf}$, $u_{sup}$ given in \eqref{eq:unif-bounds} and $\widetilde C_0$ as in  \eqref{eq:unif-bounds-v_nu}\,. 
In particular, as $\nu\to\infty$ 
\begin{equation}\label{eq:conv-int_u_dy}
     \int_0^y u^\nu(y',t)\, dy'  ~\to~ \int_0^y u (y',t)\, dy'\qquad \forall\, y\in \mathbb{T}_M\,,\ t\ge 0\,.
 \end{equation}

 \item [b)] For every $T>0$, the map $y\mapsto (u,v)(y,\cdot) \in L^1(0,T)$ is well defined on $\mathbb{T}_M$ and Lipschitz continuous, with Lipschitz constant $\widetilde L$ possibly depending on $T$. 

 \item[c)] 
$(u,v)$ is an entropy weak solution to~\eqref{eq:system_Lagrangian_M} with inital data~\eqref{eq:init-data-uv}.
\end{itemize}
\end{theorem}

\begin{proof} 
To prove the convergence, we show that the sequence $(u^{\nu},v^{\nu})$ satisfies the assumptions of Helly's compactness theorem 
(see Th. 2.3 in \cite{Bressan_Book}) on each set $[0,M)\times [0,T]$ for every $T>0$. 

Indeed, thanks to Lemmas~\ref{lem:bounds-on-bv}--\ref{Lemmauinfsup} the quantity $\tv\{(u^{\nu},v^{\nu})(\cdot,t)\}$, 
is uniformly bounded in $t\ge 0$ and $\nu\in\N$, and the same holds for the $L^\infty$ norm of $(u^{\nu},v^{\nu})$ in $(y,t)$. 
Moreover, one can prove that
\begin{equation}\label{eq:int-y-lip-cont-t}
\left\|u^\nu\left(\cdot,t_2 \right) - u^\nu\left(\cdot,t_1 \right)\right\|_{L^1(0,M)}
\le \widetilde L \,|t_2 - t_1| \qquad 
\end{equation}
where $\widetilde L$ is the product of a global bound on $\tv\{u^\nu(\cdot,t); [0,M)\}$ and a bound $\lambda^*$ on the propagation speeds of the wave fronts. 
Thanks to Lemmas~\ref{prop:equivalence-Lin}--\ref{Lemmauinfsup}, the constant $\widetilde L$ can be chosen independently of $t$ and $\nu$. 
Hence, by Helly's theorem, there exists a subsequence of $(u^{\nu},v^{\nu})$ that converges to a limit $(u,v)(y,t)$ in $L^1_{loc}\left([0,M)\times [0,T]\right)$, for every fixed $T$. 
By choosing a sequence of times $T_\nu = 1/\sqrt{\eta_\nu}$ and recalling \eqref{eq:unif-bounds}, we can use a diagonal argument to find a subsequence of $(u^{\nu},v^{\nu})$  that converges in $L^1_{loc}\left([0,M)\times [0,+\infty)\right).$
As a consequence, the convergence in \eqref{eq:conv-int_u_dy} holds for every $y\in [0,M)$. Moreover, 
$(u,v)$ satisfies the initial condition~\eqref{eq:init-data-uv}, as well as property~\eqref{eq:Linfty-bound-uv} because of~\eqref{eq:unif-bounds} and~ \eqref{eq:unif-bounds-v_nu}. The proof of (c) is standard and, therefore, omitted. 
%
\end{proof}

{\bf Proof of Theorem~\ref{newthm}.}\quad
Theorem~\ref{newthm} follows immediately since Theorem~\ref{Th-1-lagr} is a generalization of Theorem~\ref{newthm}.

\section{Proof of Theorem~\ref{Th-1}} \label{Sect:3}
This section is devoted to the proof of Theorem~\ref{Th-1}. As a first step we introduce the change of variables, from Lagrangian to Eulerian.
Define
\begin{equation}\label{eq:def-of-inverse-chi}
    x^\nu(y,t) :=  \int_0^y u^\nu\left(y',t \right) dy'
    \,,\qquad y\in [0,M),\ t\ge 0.
\end{equation}
and consider the map 
\begin{align}\label{map}
    (y,t) &~\mapsto~ \left(x^\nu(y,t),t\right) \\
  [0,M)\times [0,1/\sqrt{\eta_\nu})  &~\mapsto~ [0,\ell^{\nu}(t))\times  [0,1/\sqrt{\eta_\nu})   \nonumber
\end{align}
where $\ell^\nu(t):=\int_0^M u^\nu(y,t) \,dy$, see also~\eqref{cal V and U}. 
\begin{lemma}\label{S3:lemma1}
	The map \eqref{map} is invertible and
\begin{align}
\label{eq:bi-Lipsch_x}
   u^\nu_{inf} |y_2-y_1| \le |x^\nu(y_2,t)- x^\nu(y_1,t)|&\le u^\nu_{sup} |y_2-y_1|\,, \qquad y_1,\ y_2\in [0,M)\,,\\[1mm]
\label{eq:xnu-Lipsch_t}
    |x^\nu(y,t_2)- x^\nu(y,t_1)| &\le \tilde L \, |t_2-t_1|\qquad \qquad  
    t_1,\ t_2 
    \in [0,1/\sqrt{\eta_\nu}]
\end{align}
for some constant $\tilde L>0$ independent of $y\in [0,M)$, $\nu$, $t$.  Moreover, for all $T>0$, $\ell^\nu(t)\to \ell$ uniformly on $[0,T]$ as $\nu\to +\infty$.
\end{lemma}

\begin{proof}
From the bounds on $u^\nu$ in \eqref{eq:unif-bounds}, for every $t\in  [0,1/\sqrt{\eta_\nu}]$ the map
\begin{align*}
y&\mapsto x^\nu(y,t)\\
[0,M) &\mapsto [0,\ell^\nu(t))
\end{align*}
is Lipschitz continuous, strictly increasing, bijective and it satisfies \eqref{eq:bi-Lipsch_x}.

Concerning \eqref{eq:xnu-Lipsch_t}, definition \eqref{eq:def-of-inverse-chi} yields
\begin{align*}
     |x^\nu(y,t_2)- x^\nu(y,t_1)| &\le  \int_0^y \left| u^\nu\left(y',t_2 \right) - u^\nu\left(y',t_1 \right)\right| dy'\\
     &\le  \widetilde L |t_2 - t_1| \;, \quad 0\le y <M,
\end{align*}
where we used \eqref{eq:int-y-lip-cont-t}.
Finally, recalling that $\ell^\nu(t)$ coincides with $\mathcal{U}^\nu(t)$ in \eqref{cal V and U},  we deduce, 
from \eqref{estimate:mathcalU} and Lemma~\ref{Lemmauinfsup}, that $\ell^\nu(t)\to \ell$ as $\nu \to +\infty,\,$ uniformly on $[0,T]$, for all $T>0$.
\end{proof}

Next, define the map 
	\begin{align*}
	(x,t) &~\to~ y=\chi^\nu(x,t)\\
	[0,\ell^\nu(t) )\times [0,1/\sqrt{\eta_\nu}) &~\mapsto~ [0,M)
	\end{align*}
	that is the inverse of $y \mapsto x^\nu(y,t)$\,, and define $\rho^\nu,\,\vv^\nu,\,\mm^\nu$ by
\begin{align}\label{eq:rho_nu}
	 \begin{cases}
	\rho^\nu(x,t)=\{u^\nu(\chi^\nu(x,t),t)\}^{-1},\\
	\vv^\nu(x,t)=v^\nu(\chi^\nu(x,t),t),\\
	m^\nu(x,t)=\rho^\nu(x,t)\vv^\nu(x,t)\,.
	\end{cases} & 
\end{align}
The bounds \eqref{eq:unif-bounds} for $u^\nu$ yield 
\begin{equation*}
    0<(u_{sup})^{-1}\le \rho^\nu(x,t) \le (u_{inf})^{-1}\,,
\end{equation*}
for $(x,t)\in [0,\ell^\nu(t) )\times [0,1/\sqrt{\eta_\nu})$.  From \eqref{eq:identity-tvlnu} and \eqref{LLin-decreases}, it follows that
\begin{equation*}
    \frac 12 \tv \{\ln(\rho^\nu)(\cdot,t); [0,\ell^\nu(t)) \} = L^\nu(t) \le L^\nu(0+)\le q\,,
\end{equation*}
and hence that 
\begin{align}\label{4. rho tv final}
     \tv \{\rho^\nu(\cdot,t); [0,\ell^\nu(t))  \}  &\le u_{sup} \cdot\tv \{\ln(\rho^\nu)(\cdot,t); [0,\ell^\nu(t))\} \le u_{sup}\cdot 2 q\,,
     \end{align}
for $t\in [0,1/\sqrt{\eta_\nu})$. Notice that the last bound is independent of $\nu$ and $t$\,. 

About $\vv^\nu$, from its definition \eqref{eq:rho_nu} it follows immediately that the bounds on the $L^\infty$ norm 
and on the total variation for $v^\nu(\cdot,t)$ are valid also for $\vv^\nu$; in particular 
\eqref{eq:unif-bound-tv-v} gives
\begin{align}\label{4. v tv final}
\tv \{\vv^\nu(\cdot,t); [0,\ell^\nu(t)) \} = \tv \{v^\nu(\cdot,t);\mathbb{T}_M\} \le C_1 \,,\qquad \text{for} \,\,t\in [0,1/\sqrt{\eta_\nu}).
\end{align}

Next, let's address the admissibility conditions and weak formulation of the approximate solutions $(\rho^\nu, \vv^\nu).$
Let $0 \leq y_1^\nu(t)<\ldots<y_{N^\nu(t)}^\nu(t)<M$ be the location of discontinuities for $(u^\nu,v^\nu)(\cdot,t)$, as in 
\eqref{eq:disc-points} and define
\begin{equation*}
    x^\nu_j(t) : = \int_0^{y_j^\nu(t)} u^\nu\left(y,t \right) dy  ~\in [0,\ell^\nu(t))\qquad j=1,\ldots,N^\nu(t).
\end{equation*}
Next lemma states that the Rankine-Hugoniot conditions are approximately satisfied across the piecewise linear curves $x^\nu_j(t)$.
The proof is totally analogous to that of \cite[Lemma 4.3]{AC2022}.

\begin{lemma}\label{lem:RH-cond-rhov}
There exists a constant $C>0$ independent of $j=1,\ldots,N^\nu(t)$, $t$ and $\nu$ such that
\begin{align*}
    \frac{d}{dt} x^\nu_j(t) &= \frac{\Delta \mm^\nu}
    {\Delta \rho^\nu}\left(x^\nu_j(t)\right) + \O(1)\eta_\nu\,, 
    \\
    \frac{d}{dt} x^\nu_j(t) &= \frac{\Delta \left( \frac{(\mm^\nu)^2}{\rho^\nu}+p(\rho^\nu)\right)
    }{\Delta \mm^\nu}\left(x^\nu_j(t)\right) + \O(1)\eta_\nu\,, 
\end{align*}
and $|\O(1)|\le C\,$. Moreover, if $y_j^\nu$ is a shock (resp. a rarefaction) for ~\eqref{eq:system_Lagrangian_M}, then $x_j^\nu$ is a shock (resp. a rarefaction) for system~\eqref{eq:system2}.
\end{lemma}


To conclude the proof of Theorem~\ref{Th-1}, let $(u^{\nu}, v^{\nu})$ be a subsequence which converges, as $\nu\to\infty$, to 
$(u,v)$ in $L^1_{loc}\left([0,M)\times[0,+\infty) \right)$ as established in Theorem~\ref{Th-1-lagr} and 
let $(u,v)(y,t)$ be the entropy weak solution of \eqref{eq:system_Lagrangian_M}-\eqref{eq:init-data-uv}. 

Then let $(\rho^{\nu}, \vv^{\nu})$ be the corresponding subsequence as in \eqref{eq:rho_nu}, defined on $[0,\ell^\nu(t))\times [0,1/\sqrt{\eta_\nu})$. 
Having the uniformly bounds of $(\rho^{\nu}, \vv^{\nu})$ on total variation as shown in~\eqref{4. rho tv final}--\eqref{4. v tv final} and Lemma~\ref{S3:lemma1}, 
we deduce the convergence of the sequence as $\nu\to\infty$, to a function $(\rho,\vv)$ in $L^1_{loc}\left([0,\ell)\times [0,+\infty)\right)$. 
Then the function $(\rho,\vv)$ is extended by $\ell$-periodicity on $\mathbb{R}$. From the convergence
$$
\int_{[0,\ell^\nu(t))} \rho^\nu(x,t) \,dx \to \int_{[0,\ell)} \rho(x,t) \,dx\qquad \nu\to\infty\,,
$$
and $\int_0^{\ell^\nu(t)} 
 \rho^\nu(x,t) \,dx =\int_0^M dy=M$, we get that the total mass is conserved in time. Also, the approximate total momentum satisfies
$$
\int_0^{\ell^\nu(t)} \rho^\nu(x,t) \vv^\nu(x,t) \,dx= {\mathcal V}^\nu(t)\to 0,\qquad \text{as } \nu\to \infty\,,
$$
using estimate~\eqref{estimate:mathcalV}. Since $M_1$ is taken to be $0$ from the change of variables in Section~\ref{S2}, 
the limit $(\rho,\vv)$ conserves also momentum on $\mathbb{T}_\ell$. 

Finally, for any test function $\phi\in C^1_c\left( \mathbb{T}_\ell\times(0,\infty)\right)$ we define
	\begin{align*}
	& \RR^\nu := 
	\iint_{[0,\ell^\nu(t))\times\R_+} \left\{\rho^\nu\phi_t + \mm^\nu\phi_x \right\}\; dxdt\,,
	\\
	&
	\MM^\nu := 
	\iint_{[0,\ell^\nu(t))\times\R_+}\left\{ \mm^\nu\phi_t
	+ \left[p(\rho^\nu) + \frac{(\mm^\nu)^2}{\rho^\nu} \right]\phi_x
	-M\mm^\nu \phi \right\}\,dx dt\,.
	\end{align*}
Since $\ell^\nu(\cdot)\to\ell$ uniformly on bounded sets, and thanks to the properties stated in Lemma~\ref{lem:RH-cond-rhov}, it follows that 
	\begin{equation*}
	\lim_{\nu\to\infty} \RR^\nu= 0 = \lim_{\nu\to\infty} \MM^\nu\,.
	\end{equation*} 
Hence $(\rho,\mm)$ is a distributional solution to ~\eqref{eq:system2}. The proof that \eqref{entropy_eulerian} holds, in the sense of distributions in 
$\mathbb{T}_\ell\times [0,+\infty)$, is standard and therefore omitted.

Therefore $(\rho,\vv)$ is an entropy weak solution of the Cauchy problem to~\eqref{eq:system2} in the sense of Definition~\ref{entropy-sol} together with ${\bf (\Psi)}$ and $M_1=0$. Since the change of variables introduced in Section~\ref{S2} preserves the definition of entropy weak solution, the inverse transformation of variables leads to an  entropy weak solution for the original problem \eqref{eq:system}-\eqref{eq:initial_datum}. This completes the proof of the theorem.

{ \small
\subsection*{Acknowledgments}
D.A. and F.A.C. were partially supported by the Ministry of University and Research (MUR), Italy under the grant PRIN 2020 - Project N. 20204NT8W4, ``Nonlinear evolution PDEs, fluid dynamics and transport equations: theoretical foundations and applications" and by the INdAM-GNAMPA Project 2023, CUP E53C22001930001, ``Equazioni iperboliche e applicazioni".
F.A.C.  was partially supported by the Project ``Sistemi iperbolici: modelli di traffico veicolare e modelli idrodinamici di tipo flocking'', 
Progetti di Ateneo per Avvio alla Ricerca 2023, University of L'Aquila.\\ 
C.C. was partially supported by the internal Grant (913724) with title ``Hyperbolic Conservation Laws: Theory and Applications" supported from University of Cyprus.
}


\begin{thebibliography}{10}
	
	\bibitem{ABCD_JEE_2015}
	D.~Amadori, P.~Baiti, A.~Corli, and E.~Dal~Santo.
	\newblock Global weak solutions for a model of two-phase flow with a single
	interface.
	\newblock {\em J. Evol. Equ.}, 15(3):699--726, 2015.
	
	\bibitem{AC2022}
	D.~Amadori and C.~Christoforou.
	\newblock {BV} solutions for a hydrodynamic model of flocking-type with
	all-to-all interaction kernel.
	\newblock {\em Math. Models Methods Appl. Sci.}, 32(11):2295--2357, 2022.
	
	\bibitem{AC2024}
	D.~Amadori and C.~Christoforou.
	\newblock Unconditional flocking for weak solutions to self-organized systems
	of {E}uler-type with all-to-all interaction kernel.
	\newblock {\em Nonlinear Anal.}, 245:Paper No. 113576, 22, 2024.
	
	\bibitem{AG_2001}
	D.~Amadori and G.~Guerra.
	\newblock Global {BV} solutions and relaxation limit for a system of
	conservation laws.
	\newblock {\em Proc. Roy. Soc. Edinburgh Sect. A}, 131(1):1--26, 2001.
	
	\bibitem{Bressan_Book}
	A.~Bressan.
	\newblock {\em Hyperbolic systems of conservation laws. The one-dimensional
		Cauchy problem}, volume~20 of {\em Oxford Lecture Series in Mathematics and
		its Applications}.
	\newblock Oxford University Press, Oxford, 2000.
	
	\bibitem{CalvoColomboFrid2008}
	D.~Calvo, R.~M. Colombo, and H.~Frid.
	\newblock {$\bf L^1$} stability of spatially periodic solutions in relativistic
	gas dynamics.
	\newblock {\em Comm. Math. Phys.}, 284(2):509--535, 2008.
	
	\bibitem{Ewelina24}
	J.~A. Carrillo, G.-Q.~G. Chen, D.~Yuan, and E.~Zatorska.
	\newblock Global solutions of the one-dimensional compressible {E}uler
	equations with nonlocal interactions via the inviscid limit.
	\newblock {\em arXiv}, 2024.
	
	\bibitem{Carrillo2010}
	J.~A. Carrillo, M.~Fornasier, G.~Toscani, and F.~Vecil.
	\newblock {\em Particle, kinetic, and hydrodynamic models of swarming}, pages
	297--336.
	\newblock Birkh\"{a}user Boston, 2010.
	
	\bibitem{Choi2019}
	Y.-P. Choi.
	\newblock The global {C}auchy problem for compressible {E}uler equations with a
	nonlocal dissipation.
	\newblock {\em Math. Models Methods Appl. Sci.}, 29(01):185--207, 2019.
	
	\bibitem{Choi2024}
	Y.-P. Choi and B.-H. Hwang.
	\newblock From {BGK}-alignment model to the pressured {E}uler-alignment system
	with singular communication weights.
	\newblock {\em J. Differential Equations}, 379:363--412, 2024.
	
	\bibitem{Cucker2007}
	F.~Cucker and S.~Smale.
	\newblock Emergent behavior in flocks.
	\newblock {\em IEEE Trans. Automat. Control}, 52(5):852--862, 2007.
	
	\bibitem{Dafermos2016}
	C.~M. Dafermos.
	\newblock {\em Hyperbolic Conservation Laws in Continuum Physics}.
	\newblock Springer Berlin Heidelberg, 2016.
	
	\bibitem{Huang2006}
	F.~Huang and R.~Pan.
	\newblock Asymptotic behavior of the solutions to the damped compressible euler
	equations with vacuum.
	\newblock {\em SIAM J. Math. Anal.}, 220(1):207--233, 2006.
	
	\bibitem{Karper2013}
	T.~K. Karper, A.~Mellet, and K.~Trivisa.
	\newblock Existence of weak solutions to kinetic flocking models.
	\newblock {\em SIAM J. Math. Anal.}, 45(1):215--243, 2013.
	
	\bibitem{Karper2014}
	T.~K. Karper, A.~Mellet, and K.~Trivisa.
	\newblock Hydrodynamic limit of the kinetic cucker-smale flocking model.
	\newblock {\em Math. Models Methods Appl. Sci.}, 25(01):131--163, 2014.
	
	\bibitem{Shvydkoy2021}
	R.~Shvydkoy.
	\newblock {\em Dynamics and analysis of alignment models of collective
		behavior}.
	\newblock Ne\v{c}as Center Series. Birkh\"{a}user/Springer, Cham, [2021]
	\copyright 2021.
	
	\bibitem{Shvydkoy2020}
	R.~Shvydkoy and E.~Tadmor.
	\newblock Topologically based fractional diffusion and emergent dynamics with
	short-range interactions.
	\newblock {\em SIAM J. Math. Anal.}, 52(6):5792--5839, 2020.
	
	\bibitem{WAGNER1987118}
	D.~H. Wagner.
	\newblock Equivalence of the {E}uler and {L}agrangian equations of gas dynamics
	for weak solutions.
	\newblock {\em J. Differential Equations}, 68(1):118--136, 1987.
	
\end{thebibliography}
\end{document}